\documentclass[11pt]{article}
\usepackage{times}
\usepackage{hyperref}
\hypersetup{
     colorlinks   = true,
     linkcolor    = blue,
     citecolor    = green
}

\usepackage{amsopn,amssymb,amsthm,amsmath,bm}
\usepackage{paralist}
\usepackage{algorithm,algorithmic}
\usepackage{color}
\usepackage{graphicx,graphics}
\usepackage{comment}

\usepackage{thmtools}
\usepackage{thm-restate}
\usepackage{cleveref}
\usepackage{multirow}
\usepackage{float}
\usepackage{subfigure}

\usepackage{enumerate}
\usepackage[style=alphabetic,natbib=true,backend=bibtex]{biblatex}
\usepackage{booktabs}       

\bibliography{pca}
\usepackage{framed}

\newlength{\defbaselineskip}
\newcommand\bbN{\ensuremath{\mathbb{N}}} 
\newcommand\bbR{\ensuremath{\mathbb{R}}} 
\newcommand\bbC{\ensuremath{\mathbb{C}}} 
\newcommand\bbE{\ensuremath{\mathbb{E}}} 
\DeclareMathOperator*{\diag}{diag} 
\newcommand\ii{\mathbf{i}}
\newcommand{\Abs}[1]{\left |#1\right|}
\newcommand{\norm}[1]{{\left\|#1\right\|}}
\newcommand{\normsmall}[1]{{\|#1\|}}

\newcommand{\Prob}[1]{{\Pr}\left(#1\right)}

\newcommand\F{\ensuremath{\mathcal{F}}} %

\newtheorem{theorem}{Theorem}

\newtheorem{definition}{Definition}
\newtheorem{lemma}[theorem]{Lemma}

\newtheorem{fact}[theorem]{Fact}

\def\C{\mathcal{C}}

\newif\ifwithcomments
\withcommentstrue
\ifwithcomments
\newcommand{\xxx}[1]{{\color{red} (#1)}}
\else
\newcommand{\xxx}[1]{}
\fi

\DeclareMathOperator{\bigO}{\mathcal{O}}

\newcommand{\Expect}[1]{\mbox{}{\mathbb{E}}\left[#1\right]}
\newcommand{\Expectsmall}[1]{\mbox{}{\mathbb{E}}[#1]}
\newcommand{\Var}[1]{\mbox{}{\mathbf{Var}}\left[#1\right]}

\newcommand{\vv}[1] {\mathbf{#1}}

\def\n{\vv{n}}
\def\u{\vv{u}}
\def\v{\vv{v}}
\def\r{\vv{r}}
\def\p{\vv{p}}

\def\x{\vv{x}}

\def\w{\vv{w}}
\def\z{\vv{z}}
\def\A{\vv{A}}
\def\B{\vv{B}}
\def\C{\vv{C}}

\def\k{\vv{k}}
\def\E{\vv{E}}
\def\F{\vv{F}}
\def\P{\vv{P}}
\def\Q{\vv{Q}}
\def\R{\vv{R}}
\def\X{\vv{X}}

\def\U{\vv{U}}
\def\V{\vv{V}}
\def\R{\vv{R}}

\def\M{\vv{M}}

\def\W{\vv{W}}

\title{Accelerated Stochastic Power Iteration}

\author{
    {\scshape Christopher De Sa$^\dagger$~
    Bryan He$^\dagger$~
    Ioannis Mitliagkas$^\dag$~
    Christopher R{\'e}$^\dag$~
    Peng Xu\thanks{Corresponding Author}}
    \vspace{3.5mm}
    \\
    $^\dag$Department of Computer Science, Stanford University
    \\
    $^*$Institute for Computational and Mathematical Engineering,  Stanford University
    \vspace{3.5mm}
    \\
    \texttt{cdesa,bryanhe,imit@stanford.edu,}
    \\
    \texttt{chrismre@cs.stanford.edu, pengxu@stanford.edu}
}

\begin{document}
\maketitle

\begin{abstract}

Principal component analysis (PCA) is one of the most powerful tools in machine learning.
The simplest method for PCA, the power iteration,  
requires $\bigO(1/\Delta)$ full-data passes to recover the principal component of a matrix with eigen-gap $\Delta$.
Lanczos, a significantly more complex method, achieves an accelerated rate of $\bigO(1/\sqrt{\Delta})$ passes.
Modern applications, however, motivate methods that only ingest a subset of available data, known as the stochastic setting.
In the online stochastic setting, simple algorithms like Oja's iteration achieve the optimal sample complexity $\bigO(\sigma^2/\Delta^2)$.
Unfortunately, they are fully sequential, and also require $\bigO(\sigma^2/\Delta^2)$ iterations, far from the $\bigO(1/\sqrt{\Delta})$ rate of Lanczos.
We propose a simple variant of the power iteration with an added momentum term, that achieves both the optimal sample and iteration complexity.  
In the full-pass setting, standard analysis shows that momentum achieves the accelerated rate, $\bigO(1/\sqrt{\Delta})$.
We demonstrate empirically that naively applying momentum to a stochastic method, does not result in acceleration.
We perform a novel, tight variance analysis that reveals the ``breaking-point variance'' beyond which this acceleration does not occur.
By combining this insight with modern variance reduction techniques, we construct stochastic PCA algorithms, for the online and offline setting, that achieve an accelerated iteration complexity $\bigO(1/\sqrt{\Delta})$.
Due to the embarassingly parallel nature of our methods, this acceleration translates directly to wall-clock time if deployed in a parallel environment.
Our approach is very general, and applies to many non-convex optimization problems that can now be accelerated using the same technique.

\end{abstract} 
\section{Introduction}
\label{sec: intro}

Principal Component Analysis (PCA) is a fundamental tool for data processing and visualization
in machine learning and statistics~\cite{hotelling1933analysis,jolliffe2002principal}. 
PCA captures variable interactions in a high-dimensional dataset by identifying the directions of highest variance: the {\em principal components}.
Standard iterative methods such as the power method and the faster Lanczos algorithm perform full passes over the data at every iteration and are effective on small and medium problems~\cite{golub2012matrix}.
Notably, Lanczos requires only $\bigO(1/\sqrt{\Delta})$ full-pass matrix-vector multiplies by the input matrix, which is optimal with respect to its eigen-gap $\Delta$ and is considered an ``accelerated rate'' compared to power method's $\bigO(1/\Delta)$ passes.

Modern machine learning applications, however, are prohibitively large for full-pass methods.
Instead, practitioners use {\em stochastic methods}:
algorithms that only ingest a random subset of the available data at every iteration.
Some methods are proposed for the so-called {\em offline}, or {\em finite-sample} setting, where the algorithm is given random access to a finite set of samples, and thus could potentially use a full-pass periodically \cite{shamir2015stochastic}.
Others are presented for the truly-stochastic or {\em online} setting, where the samples are randomly drawn from a distribution, and full passes are not possible \cite{mitliagkas2013memory,boutsidis2015online,jain2016matching}.
Information theoretic bounds \cite{allen2016first} show that,
the number of samples necessary to recover the principal component, known as the {\em sample complexity},
 is at least $\bigO(\sigma^2/\Delta^2)$ in the online setting.
A number of elegant variants of the power method
 have been shown to match this lower bound in various regimes \cite{jain2016matching,allen2016first}.

However, sample complexity is not a great proxy for run time. 
 \emph{Iteration complexity}---the number of outer loop iterations required, when the inner loop is embarrassingly parallel---provides an asymptotic performance measure of an algorithm on a highly parallel computer.
We would like to match Lanczos' $\bigO(1/\sqrt{\Delta})$ iterations from the full-pass setting.
Unfortunately, the Lanczos algorithm cannot operate in a stochastic setting 
and none of the simple stochastic power iteration variants achieve this accelerated  iteration complexity.
Recently, this kind of acceleration has been achieved
with carefully tuned numerical methods based on approximate matrix inversion~\cite{garber2016faster,allen2016doubly}. 
 These methods are largely theoretical in nature and significantly more complex than stochastic power iteration.
This context motivates the question:
{\em is it possible to achieve the optimal sample and iteration complexity with a method as simple as power iteration?}

In this paper, we propose a class of simple PCA algorithms based on the power method that
(1) operate in the stochastic setting,
(2) have a sample complexity with an asymptotically optimal dependence on the eigen-gap, and
(3) have an iteration complexity with an asymptotically optimal dependence on the eigen-gap
(i.e.\ one that matches the worst-case convergence rate for the Lanczos method).
As background for our method,
we first note that a simple modification of the power iteration, {\em power iteration with momentum}, achieves the accelerated convergence rate $\bigO(1/\sqrt{\Delta})$.
Our proposed algorithms come from the natural idea of designing an efficient, stochastic version of that method.

%


We first demonstrate that simply adding momentum to a stochastic method like Oja's does not always result in acceleration.
While momentum accelerates the convergence of expected iterates, variance typically dominates so no overall acceleration is observed (cf.\ Section~\ref{sec: noisy}).
Using Chebyshev polynomials to derive an exact expression for the variance of the iterates of our algorithm,
we identify the precise relationship between sample variance and acceleration.
Importantly, we identify the exact break-down point beyond which variance is too much and acceleration is no longer observed.

Based on this analysis, we show that 
we can design a stochastic version of the momentum power iteration that is guaranteed to work.
We propose two versions 
based on \emph{mini-batching} and \emph{variance reduction}.
Both of these techniques are  used to speed up computation in stochastic optimization and are embarrassingly parallel.
This property allows our method to achieve true {\em wall-clock time acceleration} even in the online setting, something not possible with state-of-the-art results.
Hence, we demonstrate that  the more complicated techniques based on approximate matrix inversion are not necessary: {\em simple momentum-based methods are sufficient to accelerate PCA}. 
Because our analysis depends only on the variance of the iterates, it is very general: it enables many non-convex problems, including matrix completion \cite{jain2013low}, phase retrieval \cite{candes2015phase} and subspace tracking~\cite{balzano2010online}, to now be accelerated using a single technique, and suggests that the same might be true for a larger class of non-convex optimization problems.

\paragraph{Our contributions}
\begin{itemize}
  \setlength\itemsep{0em}
  \item We study the relationship between variance and acceleration by finding an exact characterization of variance for a general class of power iteration variants with momentum in Section \ref{subsec: spmm}.
  \item Using this bound, we design an algorithm using mini-batches and obtain the optimal iteration and sample complexity for the online setting in Section \ref{subsec: minibatch}.
  \item We design a second algorithm using variance reduction to obtain the optimal rate for the offline setting in Section \ref{subsec: vr}.
    Notably, operating in the offline setting, we are able to use a batch size that is independent of the target accuracy.
\end{itemize}

\begin{table}[H]
\caption{ 
Asymptotic complexities for variants of the power method to achieve $\epsilon$ accuracy, $1 - (\u_1^T\w)^2 \le \epsilon$.
For momentum methods, we choose the optimal $\beta=\lambda_2^2/4$. Here $\Delta:= {\lambda_1 - \lambda_2}$ is the eigen-gap, $\sigma^2$ is the variance of one random sample and $r$ is an a.s.\ norm bound (see Definition \eqref{def}). In $\bigO$ notation, we omit the factors depending on failure probability $\delta$. \citet{jain2016matching} and \citet{shamir2015stochastic} give the best known results for stochastic PCA without and with variance reduction respectively.
However, neither of these results achieve the optimal iteration complexity. 
Furthermore, they are not tight in terms of the variance of the problem (i.e. when $\sigma^2$ is small, the bounds are loose).}
\label{tab: rate}
\centering
\begin{tabular}{llcccc}
\toprule
  Setting & Algorithm  & Number of Iterations & Batch Size & Reference  \\
\midrule
 \multirow{3}{*}{Deterministic}
  & Power                     & $\bigO\left(\frac{1}{\Delta}\cdot\log\left(\frac{1}{\epsilon}\right)\right)$ & $n$ & \cite{golub2012matrix} \\
  & Lanczos                   & $\bigO\left(\frac{1}{\sqrt\Delta} \cdot \log\left(\frac{1}{\epsilon}\right)\right)$ & $n$ & \cite{golub2012matrix} \\
  & {\bf Power+M}             & $\bigO\left(\frac{1}{\sqrt\Delta} \cdot \log\left(\frac{1}{\epsilon}\right)\right)$ & $n$ & {\bf This paper} \\
\midrule
 \multirow{2}{*}{Online}
 & Oja's                     & $\bigO\left(\frac{\sigma^2}{\Delta^2} \cdot \frac{1}{\epsilon} + \frac{1}{\sqrt\epsilon}\right)$ & $\bigO(1)$ & \cite{jain2016matching} \\ 
 & {\bf Mini-batch Power+M}  & $\bigO\left(\frac{1}{\sqrt\Delta}\cdot\log\left(\frac{1}{\epsilon}\right)\right)$ & $\bigO\left(\frac{\sqrt d\sigma^2}{\Delta^{3/2}}\cdot\frac{1}{\epsilon}\log\left(\frac{1}{\epsilon}\right)\right)$& {\bf This paper} \\ 
  \midrule
 \multirow{2}{*}{Offline}
 & VR-PCA                    & $\bigO\left(\frac{r^2}{\Delta^2} \cdot \log\left(\frac{1}{\epsilon}\right)\right)$  & $\bigO(1)$& \cite{shamir2015stochastic} \\ 
 & {\bf VR Power+M}          & $\bigO\left(\frac{1}{\sqrt{\Delta}}\cdot \log\left(\frac{1}{\epsilon}\right)\right)$ & $\bigO\left(\frac{\sqrt d \sigma^2}{\Delta^{3/2}}\right)$ & {\bf This paper} \\
\bottomrule
\end{tabular}
\end{table}

\section{Power method with momentum}
\label{sec: momentum_pca}
In this section, we describe the basic PCA setup and show that a momentum scheme can be used to accelerate the standard power method.
This momentum scheme, and its connection with the Chebyshev polynomial family, serves as the foundation of our stochastic method.

\noindent
\textbf{PCA} ~~Let $\x_1, \cdots,\x_n\in\bbR^d$ be $n$ data points. 
The goal of PCA is to find the top eigenvector of the symmetric positive semidefinite (PSD) matrix $\A=\frac{1}{n}\sum_{i=1}^n \x_i\x_i^T\in\bbR^{d\times d}$ (the sample covariance matrix) when the data points are centered at the origin.
We assume that the target matrix $\A$ has eigenvalues $1 \geq \lambda_1 > \lambda_2 \ge \cdots \lambda_d \ge 0$
with corresponding normalized eigenvectors $\u_1, \u_2, \cdots, \u_d$.
The power method estimates the top eigenvector by repeatedly applying the update step 
\vspace{-1mm}
\[
\w_{t+1} = \A \w_t
\]
with an arbitrary initial vector $\w_0\in\bbR^d$.
After $\bigO\left(\frac{1}{\Delta}\log\frac{1}{\epsilon}\right)$ steps, the normalized iterate $\w_t / \| \w_t \| $\footnote{The $\|\cdot\|$ in this paper is $\ell_2$ norm for vectors and spectral norm for matrices.} is an \emph{$\epsilon$-accurate estimate} of top principal component. Here $\epsilon$ accuracy is measured by the squared sine of the angle between $\u_1$ and $\w_t$, which is $\sin^2\angle (\u_1,\w_t) \triangleq 1 - {(\u_1^T \w_t)^2}/{\norm{\w_t}^2}$.

When $\lambda_1$ is close to $\lambda_2$ (the eigengap $\Delta$ is small),
then the power method will converge very slowly.
To address this, we propose a class of algorithms based on the alternative update step
\begin{align}
\label{alg: generic}
  \w_{t+1} = \A \w_t - \beta \w_{t-1}. \tag{\bf{A}}
\end{align}
We call the extra term, $\beta \w_{t-1}$, the \emph{momentum} term, and $\beta$ the momentum parameter, in analogy to the heavy ball method \cite{polyak1964some}, which uses the same technique to address poorly conditioned problems in convex optimization.
For appropriate settings of $\beta$, this \emph{accelerated power method} can converge dramatically faster than the traditional power method; this is not surprising, since the same is true for analogous accelerated methods for convex optimization.

\noindent
\textbf{Orthogonal polynomials}  ~~We now connect the dynamics of the update (\ref{alg: generic}) to the behavior of a family of orthogonal polynomials, which allows us to use well-known results about orthogonal polynomials to analyze the algorithm's convergence. Consider the polynomial sequence $p_t(x)$, defined as
\begin{align}
\label{alg: momentum}
p_{t+1}(x) = xp_t(x) - \beta p_{t-1}(x), p_0 = 1, p_{1} = x/2. \tag{{\bf P}}
\end{align}
According to Favard's theorem \cite{chihara2011introduction}, this recurrence forms an orthogonal polynomial family---in fact these are scaled Chebyshev polynomials.
If we use the update (\ref{alg: generic}) with appropriate initialization, then our iterates will be given by
\begin{align*}
  \textstyle
  \w_t = p_t(\A) \w_0 = \sum_{i=1}^d p_t(\lambda_i) \u_i \u_i^T \w_0.
\end{align*}
We can use this expression, together with known facts about the Chebyshev polynomials, to explicitly bound the convergence rate of the accelerated power method with the following theorem (Analysis and proof in Appendix \ref{a_sec: momentum}).
\begin{restatable}{theorem}{thmaccrate}
\label{thm: acc_rate}
Given a PSD matrix $\A\in \bbR^{n\times n}$ with eigenvalues $1\ge\lambda_1 > \lambda_2 \ge \cdots \ge \lambda_n \geq0$, running update \eqref{alg: generic} with $\lambda_2 \le 2\sqrt{\beta} < \lambda_1$ results in estimates with worst-case error
\begin{align*}
\sin^2\angle (\u_1,\w_t) \triangleq 1 - \frac{(\u_1^T \w_t)^2}{\norm{\w_t}^2}
\le \frac{4}{\Abs{\w_0^T\u_1}^2} \cdot
\left(\frac{2\sqrt\beta}{\lambda_1 + \sqrt{\lambda_1^2 - 4\beta}}\right)^{2t}.
\end{align*}
\end{restatable}
We can derive the following corollary, which gives the iteration complexity to achieve $\epsilon$ error. 
\begin{restatable}{corollary}{corconvergence}
\label{cor: convergence}
In the same setting as Theorem \ref{thm: acc_rate}, update \eqref{alg: generic} with $\w_0\in\bbR^d$ such that $\u_1^T\w_0 \ne 0$, for any $\epsilon\in(0,1)$, 
after 
$T = \bigO\left(\frac{\sqrt\beta}{\sqrt{\lambda_1^2 - 4\beta}}\cdot \log\frac{1}{\epsilon}\right)$ iterations achieves $1 - \frac{(\u_1^T \w_T)^2}{\norm{\w_T}^2} \le \epsilon$.
\end{restatable}
\vspace{-2mm}
\noindent
\textbf{Remark.}
Minimizing $\frac{\sqrt\beta}{\sqrt{\lambda_1^2 - 4\beta}}$ over $[\lambda_2^2/4, \lambda_1^2/4)$
tells us that $\beta = \lambda_2^2/4$ is the optimal setting.

When we compare this algorithm to power iteration, we notice that it is converging at an accelerated rate.
In fact, as shown in Table \ref{tab: rate}, this momentum power method scheme (with the optimal assignment of $\beta=\lambda_2^2/4$) even matches the worst-case rate of the Lanczos method.

\vspace{3.6mm}
\noindent
\textbf{Extensions} ~~In Appendix \ref{subsec: block}, we extend this momentum scheme to achieve acceleration in the setting where we want to recover multiple top eigenvectors of $A$, rather than just one.
In Appendix \ref{subsec:stability} we show that this momentum method is numerically stable, whereas the Lanczos method suffers from numerical instability \cite{trefethen1997numerical, golub2012matrix}.
Next, in Appendix \ref{sec: tuning} we provide a heuristic for auto-tuning the momentum parameter, which is useful in practice.
Finally, in Appendix \ref{subsec: inhomogeneous}, we consider a larger orthogonal polynomial family, and we show that given some extra information about the tail spectrum of the matrix, we can obtain even faster convergence by using a 4-term inhomogeneous recurrence.

\section{Stochastic PCA}
\label{sec: noisy}
Motivated by the results in the previous section, we study using momentum to accelerate PCA in the stochastic setting. We consider a streaming PCA setting, where
we are given a series of i.i.d. samples, $\tilde\A_t$, 
such that
\begin{align}\label{def}
\Expectsmall{\tilde\A_t} = \A, 
~~
\max_{t} \normsmall{\tilde\A_t} \le r,
~~
\Expectsmall{\normsmall{\tilde\A_t - \A}^2} = \sigma^2.
\end{align}
In the sample covariance setting of Section \ref{sec: momentum_pca}, $\tilde\A_t$ can be obtained by selecting $\x_i\x_i^T$, where $\x_i$ is uniformly sampled from the dataset. One of the most popular streaming PCA algorithms is Oja's algorithm \cite{oja1982simplified}, which repeatedly runs the update\footnote{Here we consider a constant step size scheme, in which the iterate will converge to a noise ball. The size of the noise ball depends on the variance.}
$\w_{t+1} = (I + \eta \tilde\A_t)\w_t$.
A natural way to try to accelerate Oja's algorithm is to directly add a momentum term, which leads to 
\begin{align}
\label{eqn:OjasMomentum}
\w_{t+1} = (I + \eta\tilde\A_t)\w_t - \beta\w_{t-1}.
\end{align}
In expectation, this stochastic recurrence behaves like the deterministic three-term recurrence (\ref{alg: generic}), which can achieve acceleration with proper setting of $\beta$.  However, we observe empirically that (\ref{eqn:OjasMomentum}) usually does not give acceleration. In Figure \ref{fig: no_acc}, we see that while adding momentum does accelerate the convergence to the noise ball, it also increases the size of the noise ball---and decreasing the step size to try to compensate for this roughly cancels out the acceleration from momentum.
This same counterintuitive phenomenon has independently been observed in \citet{goh2017why} for stochastic optimization.
The inability of momentum to accelerate Oja's algorithm is perhaps not surprising because the sampling complexity of Oja's algorithm is asymptotically optimal in terms of the eigen-gap~\cite{allen2016first}.

In Section \ref{sec: convergence}, we will characterize this connection between the noise ball size and momentum in more depth by presenting an exact expression for the variance of the iterates.
Our analysis shows that when the sample variance is bounded, momentum can yield an accelerated convergence rate.
In this section, we will present two methods that can be used to successfully control the variance: mini-batching and variance reduction.
A summary of our methods and convergence rates is presented in Table~\ref{tab: rate}.

\begin{figure}[t]
\centering
\subfigure[Oja's algorithm with Momentum]{
\includegraphics[width=0.32\textwidth]{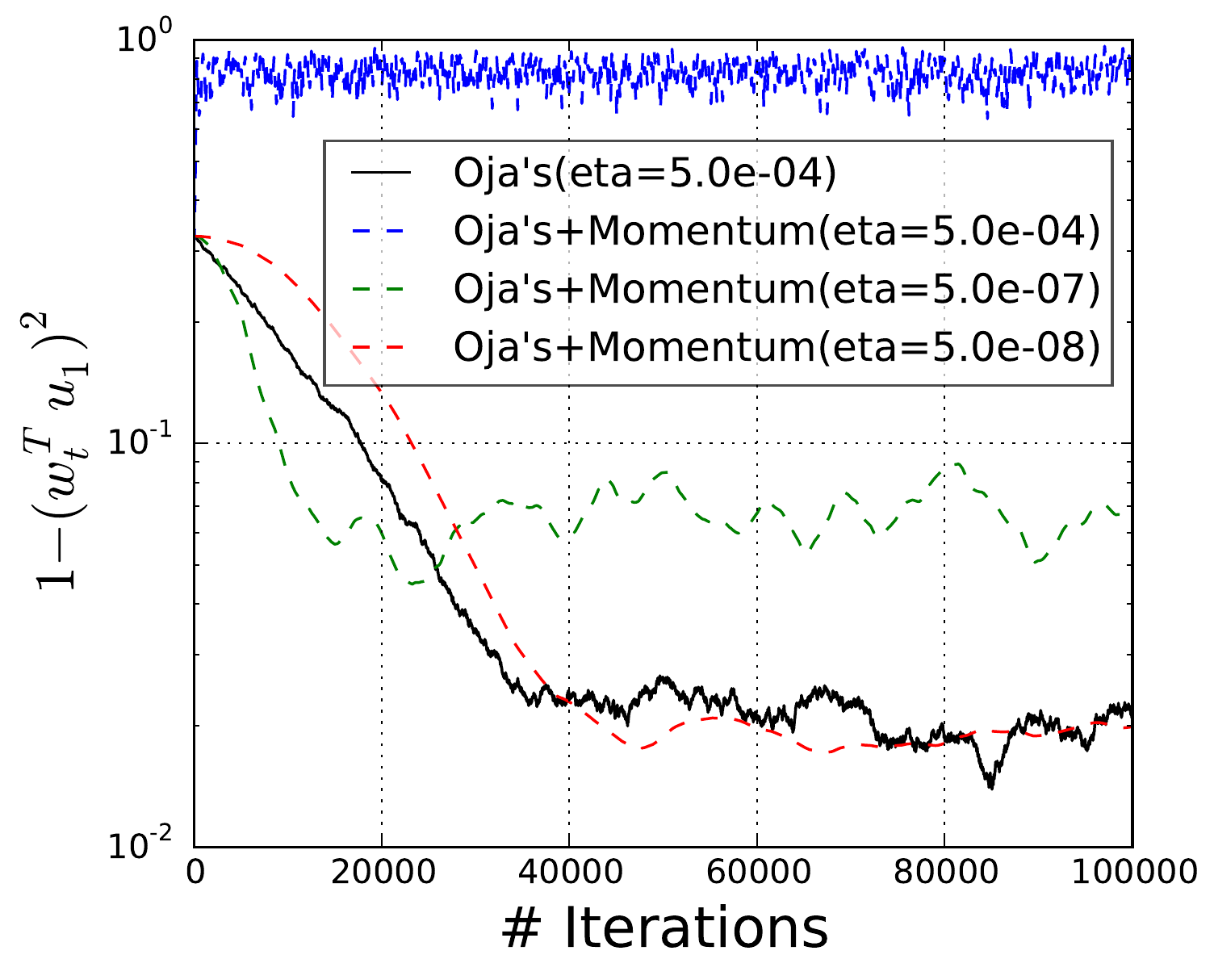}\label{fig: no_acc}
}
\subfigure[Without Variance Reduction]{
\includegraphics[width=0.315\textwidth]{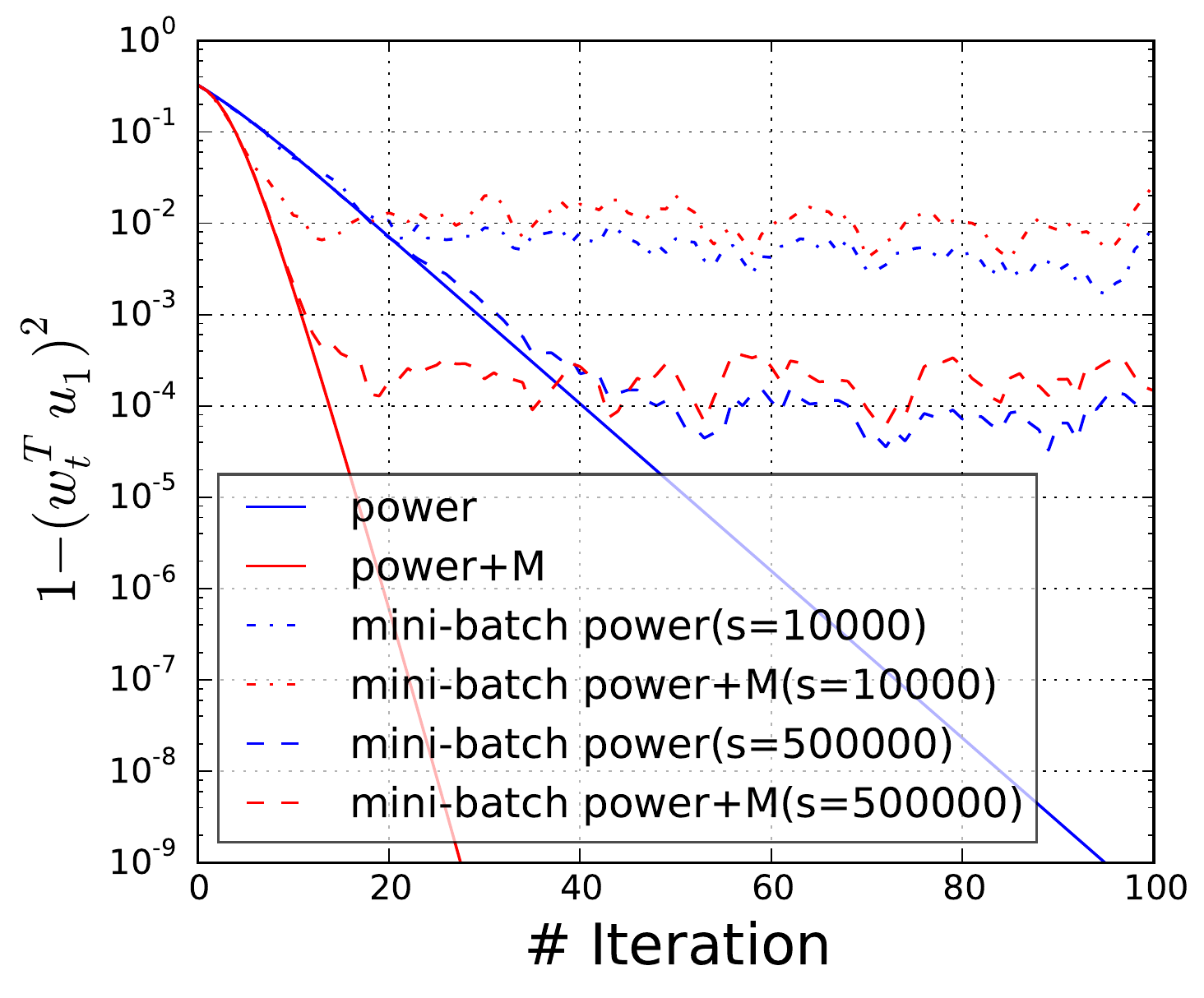}\label{fig: mini-batch}}
\subfigure[With Variance Reduction]{
\includegraphics[width=0.315\textwidth]{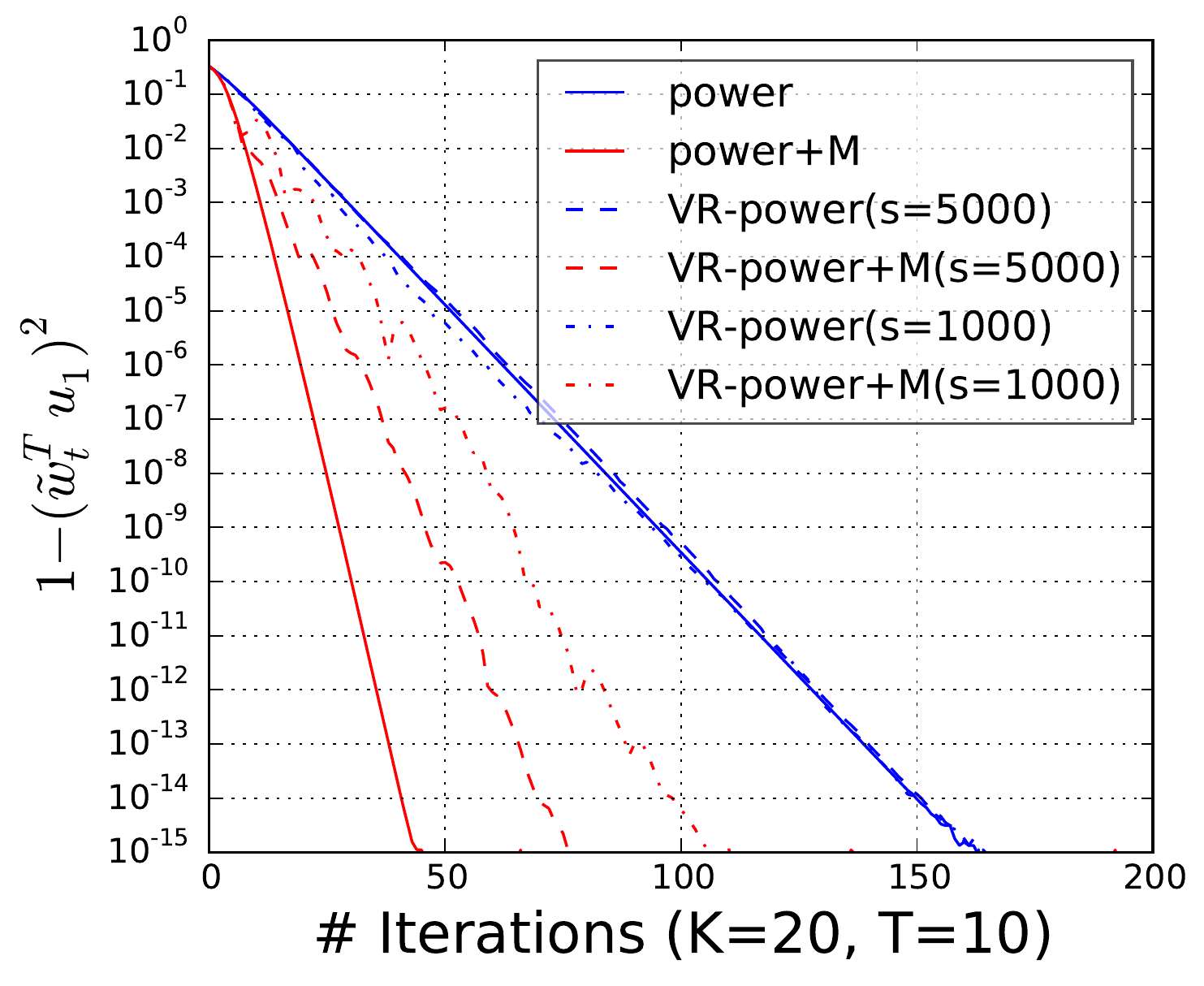}\label{fig: var}}
\caption{Different PCA algorithms on a synthetic dataset $\X\in\bbR^{10^6\times 10}$ where the covariance matrix has eigen-gap $\Delta = 0.1$. Figure \ref{fig: no_acc} shows the performance of Oja's algorithm with momentum. The momentum is set to the optimal $\beta = (1 + \eta\lambda_2)^2/4$. Different dashed lines correspond to different step sizes ($\beta$ changes correspondingly) for momentum methods. 
Figure \ref{fig: mini-batch} shows the performance of mini-batch power methods. Increasing the mini-batch size led to a smaller noise ball. Figure \ref{fig: var} shows the performance of VR power methods. The epoch length $T=10$ was estimated according to \eqref{eq: T} by setting $\delta=1\%$ and $c = 1/16$. Stochastic methods report the average performance over 10 runs.}
\label{fig: 1}
\end{figure}

\subsection{Stochastic power method with momentum}
\label{subsec: spmm}
In addition to adding momentum to Oja's algorithm, another natural way to try to accelerate stochastic PCA is to use the deterministic update \eqref{alg: generic} with random samples $\tilde\A_t$ rather than the exact matrix $\A$.
Specifically, we analyze the stochastic recurrence
\begin{align}\label{alg: mini_pca}
\w_{t+1} = \A_t\w_t - \beta \w_{t-1},
\end{align}
where $\A_t$ is an i.i.d.\ unbiased random estimate of $\A$.
More explicitly, we write this as Algorithm~\ref{alg: mini_power_m}.

\begin{algorithm}[H]
  \caption{Mini-batch Power Method with Momentum ({\bf Mini-batch Power+M})}
  \begin{algorithmic}
    \REQUIRE Initial point $\w_0$, Number of Iterations $T$, Batch size $s$, Momentum parameter $\beta$
    \STATE $\w_{-1} \gets \mathbf{0}$,
    \FOR {t = 0 \TO T - 1}
      \STATE Generate a mini-batch of i.i.d. samples $B=\{\tilde\A_{t_1},\cdots,\tilde\A_{t_s}\}$
      \STATE Update: $\w_{t+1} \gets ( \frac{1}{s}\sum_{i=1}^s \tilde\A_{t_i} ) \w_t - \beta \w_{t-1}$
      \STATE Normalization: $\w_{t} \gets \w_t / \norm{\w_{t+1}}, \w_{t+1}\gets {\w_{t+1}} / {\norm{\w_{t+1}}}$.
    \ENDFOR
    \RETURN $\w_T$
  \end{algorithmic}
\label{alg: mini_power_m}
\end{algorithm}
When the variance is zero, the dynamics of this algorithm are the same as the dynamics of update \eqref{alg: generic}, so it converges at the accelerated rate given in Theorem~\ref{thm: acc_rate}.
Even if the variance is nonzero, but sufficiently small, we can still prove that Algorithm~\ref{alg: mini_power_m} converges at an accelerated rate.
\begin{restatable}{theorem}{thmnonsvrg}\label{thm: nonsvrg}
Suppose we run Algorithm \ref{alg: mini_power_m} with $2\sqrt\beta\in[\lambda_2,\lambda_1)$. Let $\Sigma = \Expect{(\A_t -\A)\otimes(\A_t -\A)}$\footnote{$\otimes$ denotes the Kronecker product.}. Suppose that $\norm{\w_0}= 1$ and $\Abs{\u_1^T\w_0}\ge 1/2$.  For any $\delta\in(0,1)$ and $\epsilon\in(0,1)$, if 
\vspace{-2mm}
\begin{align}\label{cond: var}
T &= \frac{\sqrt\beta}{\sqrt{\lambda_1^2 - 4\beta}}\log\left(\frac{32}{\delta\epsilon}\right),~~~
\norm{\Sigma} \le \frac{(\lambda_1^2 - 4\beta)\delta\epsilon}{256 \sqrt d T} =  \frac{(\lambda_1^2 - 4\beta)^{3/2}\delta\epsilon}{256\sqrt  d \sqrt\beta}\log^{-1}\left(\frac{32}{\delta\epsilon}\right),
\end{align}
then with probability at least $1 - 2\delta$, we have  $1 - (\u_1^T\w_T)^2 \le \epsilon.$
\end{restatable}
When we compare this to the result of Theorem~\ref{thm: acc_rate}, we can see that as long as the variance $\| \Sigma \|$ is sufficiently small, the number of iterations we need to run in the online setting is the same as in the deterministic setting (up to a constant factor that depends on $\delta$).
In particular, this is faster than the power method without momentum in the deterministic setting.
Of course, in order to get this accelerated rate, we need some way of getting samples that satisfy the variance condition of Theorem~\ref{thm: nonsvrg}.
Certain low-noise datasets might satisfy this condition, but this is not always the case.
In the next two sections, we discuss methods of getting lower-variance samples.

\subsection{Controlling variance with mini-batches}
\label{subsec: minibatch}
In the online PCA setting, a natural way of getting lower-variance samples is to increase the \emph{batch size} (parameter $s$) used by Algorithm~\ref{alg: mini_power_m}. 
Using the following bound on the variance,
\[
\norm{\Sigma} =\norm{\Expect{(\A_t - \A) \otimes (\A_t - \A)}} \le \Expect{\norm{(\A_t - \A) \otimes (\A_t - \A)}} = \Expect{\norm{\A_t-\A}^2} = \frac{\sigma^2}{s},
\]
we can get an upper bound on the mini-batch size we will need in order to satisfy the variance condition in Theorem \ref{thm: nonsvrg},
which leads to the following corollary.
\begin{restatable}{corollary}{cornonsvrg}\label{cor: nonsvrg}
Suppose we run Algorithm \ref{alg: mini_power_m} with $2\sqrt\beta\in[\lambda_2,\lambda_1)$. Assume that $\norm{\w_0}= 1$ and $\Abs{\u_1^T\w_0}\ge 1/2$.  For any $\delta\in(0,1)$ and $\epsilon\in(0,1)$, if  
\begin{align*}
T = \frac{\sqrt\beta}{\sqrt{\lambda_1^2 - 4\beta}}\log\left(\frac{32}{\delta\epsilon}\right),~~~ 
s \ge \frac{256\sqrt d \sigma^2 T}{(\lambda_1^2 - 4\beta)\delta\epsilon} =  \frac{256\sqrt d \sqrt\beta\sigma^2}{(\lambda_1^2 - 4\beta)^{3/2}\delta\epsilon}\log\left(\frac{32}{\delta\epsilon}\right),
\end{align*}
then with probability at least $1 - 2\delta$, $1 - (\u_1^T\w_T)^2 \le \epsilon.$
\end{restatable}

This means that no matter what the variance of the estimator is, we can still converge at the same rate as the deterministic setting as long as we can compute mini-batches of size $s$ quickly.
One practical way of doing this is by using many parallel workers: a mini-batch of size $s$ can be computed in $\bigO(1)$ time by $\bigO(s)$ machines working in parallel.
If we use a sufficiently large cluster that allows us to do this, this means that Algorithm~\ref{thm: acc_rate} converges in asymptotically less time than any non-momentum power method that uses the cluster for mini-batching, because we converge faster than even the deterministic non-momentum method.

One drawback of this approach is that the required variance decreases as a function of $\epsilon$, so we will need to increase our mini-batch size as the desired error decreases.
If we are running in parallel on a cluster of fixed size, this means that we will eventually exhaust the parallel resources of the cluster and be unable to compute the mini-batches in asymptotic $\bigO(1)$ time.
As a result, we now seek methods to reduce the required batch size, and remove its dependence on $\epsilon$.

\subsection{Reducing batch size with variance reduction}
\label{subsec: vr}
Another way to generate low-variance samples is the \emph{variance reduction} technique.
This technique can be used if we have access to the target matrix $\A$ so that we can occasionally compute an exact matrix-vector product with $\A$.
For example, in the offline setting, we can compute $\A\w$ by occasionally doing a complete pass over the data.
In PCA, \citet{shamir2015stochastic} has applied the standard variance reduction technique that was used in stochastic convex optimization \cite{johnson2013accelerating}, in which the stochastic term in the update is 
\begin{align}\label{eq: vr1}
\A\w_t + (\A_t - \A)(\w_t - \tilde\w),
\end{align}
where $\tilde\w$ is the (normalized) anchor iterate, for which we know the exact value of $\A\tilde\w$.
We propose a slightly different variance reduction scheme, where the stochastic term in the update is 
\begin{align}\label{eq: vr2}
\left[\A + (\A_t - \A)(I -\tilde\w\tilde\w^T)\right] \w_t = \A\w_t + (\A_t - \A)(I -\tilde\w\tilde\w^T)\w_t.
\end{align}
It is easy to verify that both \eqref{eq: vr1} and \eqref{eq: vr2} can be computed using only the samples $\A_t$ and the exact value of $\A\tilde\w$.
In the PCA setting, \eqref{eq: vr2} is more appropriate because progress is measured by the angle between $\w_t$ and $\u_1$, not the $\ell_2$ distance as in the convex optimization problem setting: this makes \eqref{eq: vr2} easier to analyze.
In addition to being easier to analyze, our proposed update rule \eqref{eq: vr2} produces updates that have generally lower variance because for all unit vectors $\w_t$ and $\tilde\w$, $\normsmall{\w_t - \tilde\w} \ge \normsmall{(I -\tilde\w\tilde\w^T)\w_t}$.
Using this update step results in the variance-reduced power method with momentum in Algorithm \ref{alg: vr_power_m}.
\begin{algorithm}[H]
  \caption{VR Power Method with Momentum ({\bf VR Power+M})}
  \begin{algorithmic}
    \REQUIRE Initial point $\w_0$, Number of Iterations $T$, Batch size $s$, Momentum parameter $\beta$
    \STATE $\w_{-1} \gets \mathbf{0}$
    \FOR {k = 1 \TO K}
    \STATE $\tilde \v \gets \A\tilde\w_{k}$ \; (Usually there is no need to materialize $\A$ in practice).
    \FOR {t = 1 \TO T}
      \STATE Generate a mini-batch of i.i.d. samples $B=\{\tilde\A_{t_1},\cdots,\tilde\A_{t_s}\}$
      \STATE Update: $~\alpha \gets \w_t^T\tilde\w_k,
      ~~~~~\w_{t+1} \gets \frac{1}{s}\sum_{i=1}^s\tilde\A_{t_i}(\w_t -\alpha  \tilde \w_k) + \alpha \tilde \v - \beta \w_{t-1}
$
      \STATE Normalization: $\w_{t} \gets \w_t / \norm{\w_{t+1}}, \w_{t+1}\gets \w_{t+1} / \norm{\w_{t+1}}$.
    \ENDFOR
    \STATE $\tilde \w_{k+1} \gets \w_{T}$.
    \ENDFOR
    \RETURN $\w_K$
  \end{algorithmic}
\label{alg: vr_power_m}
\end{algorithm}
\vspace{-2mm}
A number of methods use this kind of SVRG-style variance reduction technique, which converges at a linear rate and is not limited by a noise ball. 
Our method improves upon that by achieving the {\em accelerated rate} throughout, and only using a  mini-batch size that is constant with respect to $\epsilon$.
\begin{restatable}{theorem}{thmsvrg}\label{thm: svrg}
Suppose we run Algorithm \ref{alg: vr_power_m} with $2\sqrt\beta \in [\lambda_2, \lambda_1)$ and a initial unit vector $\w_0$ such that $1 - (\u_1^T\w_0)^2 \le \frac{1}{2}$. For any $\delta, \epsilon\in(0,1)$, if 
\vspace{-2mm}
\begin{align}\label{eq: T}
T = \frac{\sqrt\beta}{\sqrt{\lambda_1^2 - 4\beta}}\log\left(\frac{1}{c\delta}\right),~~~~
s \ge \frac{32\sqrt d\sqrt\beta \sigma^2}{c(\lambda_1^2 - 4\beta)\delta} \log\left(\frac{1}{c\delta}\right),
\end{align}
then after $K = \bigO\left(\log(1/\epsilon)\right)$ epochs, with probability at least $1 - \log\left(\frac{1}{\epsilon}\right)\delta$, we have $1 - (\u_1^T\tilde\w_K)^2 \le \epsilon$, 
where $c\in(0,1/16)$ is a numerical constant.
\end{restatable}
By comparing to the results of Theorem~\ref{thm: acc_rate} and Theorem~\ref{thm: svrg}, we notice that we still achieve the same convergence rate, in terms of the total number of iterations we need to run, as the deterministic setting.
Compared with the non-variance-reduced setting, notice that the mini-batch size we need to use does not depend on the desired error $\epsilon$, which allows us to use a fixed mini-batch size throughout the execution of the algorithm.
This means that we can use Algorithm~\ref{alg: vr_power_m} together with a parallel mini-batch-computing cluster of fixed size to compute solutions of arbitrary accuracy at a rate faster than any non-momentum power method could achieve.
As shown in Table \ref{tab: rate}, in terms of number of iterations, the momentum methods achieve accelerated linear convergence with proper mini-batching (our results there follow Corollary \ref{cor: nonsvrg} and Theorem \ref{thm: svrg}, using the optimal momentum $\beta = \lambda_2^2 / 4$.).  

\noindent
\textbf{Experiment} ~~Now we use some simple synthetic experiments (details in Appendix \ref{a_sec: data}) to illustrate how the variance affects the momentum methods. Figure \ref{fig: mini-batch} shows that the stochastic power method maintains the same linear convergence as the deterministic power method before hitting the noise ball. Therefore, the momentum method can accelerate the convergence before hitting the noise ball.
Figure \ref{fig: var} shows that the variance-reduced power method indeed can achieve an accelerated linear convergence with a much smaller batch size on this same synthetic dataset.


\section{Convergence analysis}\label{sec: convergence}
In this section, we sketch the proofs of Theorems \ref{thm: nonsvrg} and \ref{thm: svrg}.
The main idea is to use Chebyshev polynomials to tightly bound the variance of the iterates. 
Either with or without variance reduction, the dynamics of the stochastic power method with momentum from (\ref{alg: mini_pca}) can be written as $\w_t = \F_t \w_0 / \norm{\F_t \w_0}$ where $\{ \F_t \}$ is a sequence of stochastic matrices in $\bbR^{d\times d}$ satisfying
\begin{align}\label{eq: recur}
\F_{t+1} = \A_{t+1} \F_t - \beta\F_{t-1}, \; \F_0 = I, \; \F_{-1} = \mathbf 0.
\end{align}
The random matrix $\A_t \in \bbR^{d\times d}$ will have different forms in Algorithm \ref{alg: mini_power_m} and Algorithm \ref{alg: vr_power_m}, but will still be i.i.d. and satisfy  $\Expect{\A_t} = \A$.
In fact this recurrence \eqref{eq: recur} is general enough to be applied into many other problems, including least-square regression and the randomized Kaczmarz algorithm \cite{strohmer2009randomized}, as well as some non-convex matrix problems~\cite{de2014global} such as matrix completion~\cite{jain2013low}, phase retrieval~\cite{candes2015phase} and subspace tracking~\cite{balzano2010online}.  
Since $\F_t$ obeys a linear recurrence, its second moment also follows a linear recurrence (in fact all its moments do).
Decomposing this recurrence using Chebyshev polynomials, we can get a tight bound on the covariance of $\F_t$, which is shown in Lemma \ref{lem: covbound}.
It is worth mentioning that this bound is exact in the scalar case. 
\begin{restatable}{lemma}{lemnonsvrgcov}\label{lem: covbound}
  Suppose $\lambda_1^2 \ge 4 \beta$ and $\Sigma = \Expect{(\A_t - \A)\otimes(\A_t - \A)}$.  The norm of the covariance of the matrix $\F_t$ is bounded by
  \[
    \norm{ \Expect{\F_t \otimes \F_t} - \Expect{\F_t} \otimes \Expect{\F_t} }
    \le
    \sum_{n = 1}^t
    \norm{\Sigma}^n
    \beta^{t - n}
    \sum_{\k \in S_{t - n}^{n + 1}}
    \prod_{i=1}^{n+1}
    U_{k_i}^2\left(\frac{\lambda_1}{2 \sqrt{\beta}} \right),
  \]
  where $U_k(\cdot)$ is the Chebyshev polynomial of the second kind, and $S_{m}^n$ denotes the set of vectors in $\bbN^n$ with entries that sum to $m$, i.e.
\[
\textstyle
S_m^n = \{ \k = (k_1,\cdots, k_n)\in\bbN^n \mid \sum_{i=1}^n k_i = m\}.
\] 
\end{restatable}
For the mini-batch power method without variance reduction (Algorithm \ref{alg: mini_power_m}), the goal is to bound $1 - (\u_1^T\w_t)^2$, which is equivalent to bounding $\sum_{i=2}^d(\u_i^T\F_t\w_0)^2 / (\u_1^T\F_t\w_0)^2$. 
We use Lemma \ref{lem: covbound} to get a variance bound for the denominator of this expression, which is 
\begin{align}\label{eq: upperbound}
\Var{\u_1^T\F_t\w_0} \le p_t^2(\lambda_1;\beta) \cdot \frac{8\norm{\Sigma} t}{\lambda_1^2 - 4\beta}.
\end{align}
With this variance bound and Chebyshev's inequality we get a probabilistic lower bound for $| \u_1^T\F_t\w_0 |$.
Lemma \ref{lem: covbound} can also be used to get an upper bound for the numerator, which is
\begin{align}
\Expect{\sum_{i=2}^d (\u_i^T\F_t\w_0)^2} \le  p_t^2(\lambda_1;\beta) \cdot \left(\frac{8\sqrt d \norm{\Sigma} t}{(\lambda_1^2 - 4\beta)} + \frac{p_t^2(2\sqrt\beta;\beta)}{p_t^2(\lambda_1;\beta)}\right)
\end{align}
By Markov's inequality we can get a probabilistic upper bound for $\sum_{i=2}^d (\u_i^T\F_t\w_0)^2$. The result in Theorem \ref{thm: nonsvrg} now follows by a union bound. The details of the proof appear in Appendix \ref{a_subsec: mini_pca}. 

Next we consider the case with variance reduction (Algorithm \ref{alg: vr_power_m}). The analysis contains two steps. The first step is to show a geometric contraction for a single epoch, i.e.
\begin{align}
\label{eqn:SVRGrho}
1 - (\u_1^T\w_T)^2 \le \rho \cdot \left(1 - (\u_1^T\w_0)^2\right),
\end{align}
with probability at least $1 - \delta$, where $\rho < 1$ is a numerical constant.
Afterwards, the second step is to get the final $\epsilon$ accuracy of the solution, which trivially requires $\bigO\left(\log (1/\epsilon)\right)$ epochs.
Thus, the analysis boils down to analyzing a single epoch. Notice that in this setting, 
\vspace{-2mm}
\begin{equation}
\label{eqn:AtSVRG}
\textstyle
\A_{t+1} = \A + \left(\frac{1}{s}\sum_{i=1}^s \tilde\A_{t_i} - \A \right)(I - \w_0\w_0^T),
\end{equation}
and again $\w_t = \F_t\w_0 / \norm{\F_t\w_0}$.
Using similar techniques to the mini-batch power method setting, we can prove a variant of Lemma~\ref{lem: covbound} that is specialized to (\ref{eqn:AtSVRG}).
\begin{restatable}{lemma}{lemsvrgvar}\label{lem: svrgvar}
  Suppose $\lambda_1^2 \ge 4 \beta$. Let $\w_0\in\bbR^d$ be a unit vector, $\theta = 1 - (\u_1^T\w_0)^2$, and
\vspace{-1mm}
  \[
    \textstyle
    \Sigma = \Expect{\left(\frac{1}{s}\sum_{i=1}^s \tilde \A_{t_i} - \A\right) \otimes \left(\frac{1}{s}\sum_{i=1}^s \tilde \A_{t_i} - \A\right)}.
  \]
\vspace{-2mm}
  Then, the norm of the covariance will be bounded by
\vspace{-2mm}
  \[
    \norm{ \Expect{\F_t \w_0 \otimes \F_t \w_0} - \Expect{\F_t \w_0} \otimes \Expect{\F_t \w_0} }
    \le 4\theta\cdot
    \sum_{n = 1}^t
    \norm{\Sigma}^n
    \beta^{t - n}
    \sum_{\k \in S_{t - n}^{n + 1}}
    \prod_{i=1}^{n+1}
    U_{k_i}^2\left(\frac{\lambda_1}{2 \sqrt{\beta}} \right).
  \]
\end{restatable}
\vspace{-2mm}
Comparing to the result in Lemma \ref{lem: covbound}, this lemma shows that the covariance is also controlled by the angle between $\u_1$ and $\w_0$ which is the anchor point in each epoch. Since the anchor point $\tilde \w_k$ is approaching $\u_1$, the norm of the covariance is shrinking across epochs---this allows us to prove (\ref{eqn:SVRGrho}).
From here, the proof of Theorem~\ref{thm: svrg} is similar to non-VR case, and the details are in Appendix~\ref{a_subsec: vrpca}.

\section{Related work}
\label{sec: related}
\paragraph{PCA}
A recent spike in research activity
has focused on improving a number of computational and statistical aspects of PCA,
including tighter sample complexity analysis~\cite{jain2016matching},
global convergence \cite{de2014global,allen2016first},
memory efficiency~\cite{mitliagkas2013memory}
and doing online
regret analysis~\cite{boutsidis2015online}.
Some work has also focused on tightening the analysis of power iteration and Krylov methods to provide gap-independent results using polynomial-based analysis techniques~\cite{musco2015randomized}.
However, that work does not consider the stochastic setting.
Some works that study Oja's algorithm~\cite{oja1982simplified} or stochastic power methods in the stochastic setting focus on the analysis of a gap-free convergence rate for the {\em distinct PCA formulation of maximizing explained variance} (as opposed to recovering the strongest direction)~\cite{shamir2016convergence, allen2016first}.
Others provide better dependence on the dimension of the problem~\cite{jain2016matching}.
\citet{garber2016faster, allen2016lazysvd} use faster linear system solvers to speed up PCA algorithms such that the convergence rate has the square root dependence on the eigengap in the offline setting. However their methods require solving a series of linear systems, which is not trivially parallelizable. 
Also none of these results give a convergence analysis that is asymptotically tight in terms of variance, which allows us to show an accelerated linear rate in the stochastic setting. 
Another line of work has focused on variance control for PCA in the stochastic setting~\cite{shamir2015stochastic} to get a different kind of acceleration.
Since this is an independent source of improvement, these methods can be further accelerated using our momentum scheme. 

\paragraph{Stochastic acceleration}
The momentum scheme is a common acceleration technique in convex optimization~\cite{polyak1964some,nesterov1983method},
and has been widely adopted as the de-facto optimization method for non-convex objectives in deep learning~\cite{sutskever2013importance}.
Provably accelerated stochastic methods have previously been found for convex problems \cite{cotter2011better,jain2017accelerating}.
However, similar results for non-convex problems remain elusive, despite empirical evidence that momentum results in acceleration for some non-convex problems \cite{sutskever2013importance,kingma2014adam}.


\vspace{-3mm}
\paragraph{Orthogonal Polynomials}
The Chebyshev polynomial family is a sequence of orthogonal polynomials \cite{chihara2011introduction} that has been used for analyzing accelerated methods.
For example, Chebyshev polynomials have been studied to accelerate the solvers of linear systems \cite{golub1961chebyshev,golub2012matrix} and to accelerate convex optimization \cite{scieur2016regularized}.
\citet{trefethen1997numerical} use Chebyshev polynomials to show that the Lanczos method is quadratically faster than the standard power iteration, which is conventionally considered as the accelerated version of power method with momentum \cite{hardt2014noisy}.
\vspace{-2mm}

\section{Conclusion}
\label{sec: conclusion}
This paper introduced a very simple accelerated PCA algorithm that works in the stochastic setting.
As a foundation, we presented the power method with momentum, an accelerated scheme in the deterministic setting.
We proved that the power method with momentum obtains quadratic acceleration like the convex optimization setting. 
Then, for the stochastic setting, we introduced and analyzed the stochastic power method with momentum.
By leveraging the Chebyshev polynomials, we derived a convergence rate that is asymptotically tight in terms of the variance.
Using a tight variance analysis, we demonstrated how the momentum scheme behaves in a stochastic system, which can lead to a better understanding of how momentum interacts with variance in stochastic optimization problems~\cite{goh2017why}.
Specifically, with mini-batching, the stochastic power method with momentum can achieve accelerated convergence to the noise ball.
Alternatively, using variance reduction, accelerated convergence at a linear rate can be achieved with a much smaller batch size.

\paragraph{Acknowledgments.}
We thank Aaron Sidford for helpful discussion and feedback on this work.

We gratefully acknowledge the support of the Defense Advanced Research Projects Agency (DARPA) SIMPLEX program under No. N66001-15-C-4043, D3M program under No. FA8750-17-2-0095, the National Science Foundation (NSF) CAREER Award under No. IIS- 1353606, the Office of Naval Research (ONR) under awards No. N000141210041 and No. N000141310129, a Sloan Research Fellowship, the Moore Foundation, an Okawa Research Grant, Toshiba, and Intel. Any opinions, findings, and conclusions or recommendations expressed in this material are those of the authors and do not necessarily reflect the views of DARPA, NSF, ONR, or the U.S. government.

\printbibliography

\newpage
\appendix
\section{Momentum PCA and Orthogonal Polynomials}\label{a_sec: momentum}
In this section, we prove Theorem \ref{thm: acc_rate} and give the intuition that the momentum can provide acceleration from both geometric and algebraic perspectives.

First, we restate the update \eqref{alg: generic} for power iteration with momentum,
\[
\w_{t+1} = \A\w_t - \beta\w_{t-1}. \tag{\bf A}
\]
and the correponding orthogonal polynomial sequence \eqref{alg: momentum},
\[
p_{t+1}(x) = xp_t(x) - \beta p_{t-1}(x), p_0 = 1, p_1 = x/2. \tag{\bf P}
\]
According to Lemma \ref{lem: op}, we have the expression of $p_t(x)$,
\[
p_t(x) = \begin{cases} \frac{1}{2} \resizebox{4cm}{!}{$\left[\left( \frac{x - \sqrt{x^2 - 4\beta}}{2}\right)^{t} + \left( \frac{x+\sqrt{x^2-4\beta}}{2}\right)^{t}  \right]$}, &  |x| > 2\sqrt\beta,
 \\ (\sqrt\beta)^t \cos\left (t \arccos(\frac{x}{2\beta})\right), & |x| \le 2\sqrt\beta.
\end{cases}
\]
\subsection{Proof of Theorem \ref{thm: acc_rate}}
Here we prove a slightly general result.
\begin{theorem}
Given a PSD matrix $\A\in\bbR^{d\times d}$ with eigenvalues $\lambda_1 > \lambda_2 \ge \cdots \ge \lambda_d$ with normalized eigenvectors $\u_1, \cdots, \u_d$, we run the power iteration with momentum update \ref{alg: generic} with a unit vector $\w_0\in\bbR^d$, the we have
\[
1 - \frac{(\u_1^T\w_t)^2}{\norm{\w_t}^2} \le \frac{1 - (\u_1^T\w_0)^2}{(\u_1^T\w_0)^2} \cdot \begin{cases}4\left(\frac{2\sqrt\beta}{\lambda_1 + \sqrt{\lambda_1^2 - 4\beta}}\right)^{2t}, & \lambda_2< 2\sqrt\beta \\ \left(\frac{\lambda_2 + \sqrt{\lambda_2^2 - 4\beta}}{\lambda_1 + \sqrt{\lambda_1^2 - 4\beta}}\right)^{2t}, & \lambda_2 \ge 2\sqrt\beta \end{cases}
\]
\end{theorem}
\begin{proof}
Denote $d_i = \w_0^T\u_i,$ and  $\delta^{(t)} = \max_{i=2,...,n}\dfrac{p_t^2(\lambda_i)}{p_t^2(\lambda_1)}$, then
\begin{align*}
1 - \frac{(\u_1^T\w_t)^2}{\norm{\w_t}^2} &= 1 - \frac{(\u_1^Tp_t(A)\w_0)^2}{\w_0^Tp_t(A)^2\w_0} 
= 1 - \frac{d_1^2 p_t^2(\lambda_1)}{\sum_{i=1}^d d_i^2 p_t^2(\lambda_i)} 
= \frac{\sum_{i=2}^n d_i^2p_t^2(\lambda_i)}{\sum_{i=1}^n d_i^2 p_t^2(\lambda_i)}
\\
&= \frac{\sum_{i=2}^nd_i^2p_t^2(\lambda_i)/p_t^2(\lambda_1)}{d_1^2 +\sum_{i=2}^n d_i^2p_t^2(\lambda_i)/p_t^2(\lambda_1)}
\\
&\le \frac{\sum_{i=2}^2 d_i^2}{d_1^2 }\delta^{(t)}
\end{align*}
Let's bound $\delta^{(t)}$. Denote $k$ as the smallest index such that $\lambda_k > 2\sqrt{\beta}$. Since $\lambda_1 > 2\sqrt{\beta}$, then $k \ge 1$. 
Now use Lemma \ref{lem: op}, we get
\begin{align*}
|p_t(\lambda_i)| &= \frac{1}{2}\left[\left( \frac{\lambda_i - \sqrt{\lambda_i^2 - 4\beta}}{2}\right)^{t} + \left( \frac{\lambda_i+\sqrt{\lambda_i^2-4\beta}}{2}\right)^{t}  \right], i \le k,\\
|p_t(\lambda_i)| &\le (\sqrt\beta)^t, i > k\\
\end{align*}
First, let's consider $2\le i\le k$. 
\begin{align*}
\left |\frac{p_t(\lambda_i)}{p_t(\lambda_1)}\right| & = \frac{  \left( \frac{\lambda_i - \sqrt{\lambda_i^2 - 4\beta}}{2}\right)^{t} + \left( \frac{\lambda_i+\sqrt{\lambda_i^2-4\beta}}{2}\right)^{t} }{\left( \frac{\lambda_1 - \sqrt{\lambda_i^2 - 4\beta}}{2}\right)^{t} + \left( \frac{\lambda_1+\sqrt{\lambda_1^2-4\beta}}{2}\right)^{t} }
\le  \left(\frac{\left(\lambda_i + \sqrt{\lambda_i^2 - 4\beta}\right)^{t} }{\left(\lambda_1 + \sqrt{\lambda_1^2 - 4\beta}\right)^{t}}\right)
\end{align*}
Now consider $i > k$,
\begin{align*}
\left |\frac{p_t(\lambda_i)}{p_t(\lambda_1)}\right| & = \frac{ 2\left(\sqrt\beta\right)^t}{\left( \frac{\lambda_1 - \sqrt{\lambda_i^2 - 4\beta}}{2}\right)^{t} + \left( \frac{\lambda_1+\sqrt{\lambda_1^2-4\beta}}{2}\right)^{t} }
\le  \frac{ 2\left(\sqrt\beta\right)^t}{\left( \frac{\lambda_1+\sqrt{\lambda_1^2-4\beta}}{2}\right)^{t} }
 =2 \left(\frac{2\sqrt\beta}{\lambda_1 + \sqrt{\lambda_1^2 - 4\beta}}\right)^{t}.
\end{align*}
Therefore plug in the bound for $\delta^{(t)}$ and we get the desired result.
\end{proof}

\subsection{Effect of Momentum}
\label{subsec: effect}

In this section, we explain why acceleration happens from both a geometric and algebraic perspective of the orthogonal polynomial recurrence. 
First, we show the geometric behavior of the orthogonal polynomial sequence.
We see that momentum results in a ``calm'' region, where the orthogonal polynomial sequence grows very slowly and an ``explosive'' region, where the polynomials grow exponentially fast.
We then show how the momentum controls the size of  ``calm'' region.
Second, we consider an algebraically equivalent form of the three-term recurrence in terms of an augmented matrix.
We see that power iteration with momentum is equivalent to standard power iteration on an augmented matrix and quantitatively how the momentum leads to a ``better-conditioned'' problem.
From either perspective, we get a better understanding about how our methods work.

\noindent\textbf{Regions of the Polynomial Recurrence}
Now, we demonstrate the effect of momentum on different eigenvalues.
In Figure \ref{fig:regions}, we show the values of the polynomial recurrence, which characterizes the growth of different eigenvalues for varying $\beta$.

\begin{figure}[htbp]
  \begin{center}
    \includegraphics[scale=0.7]{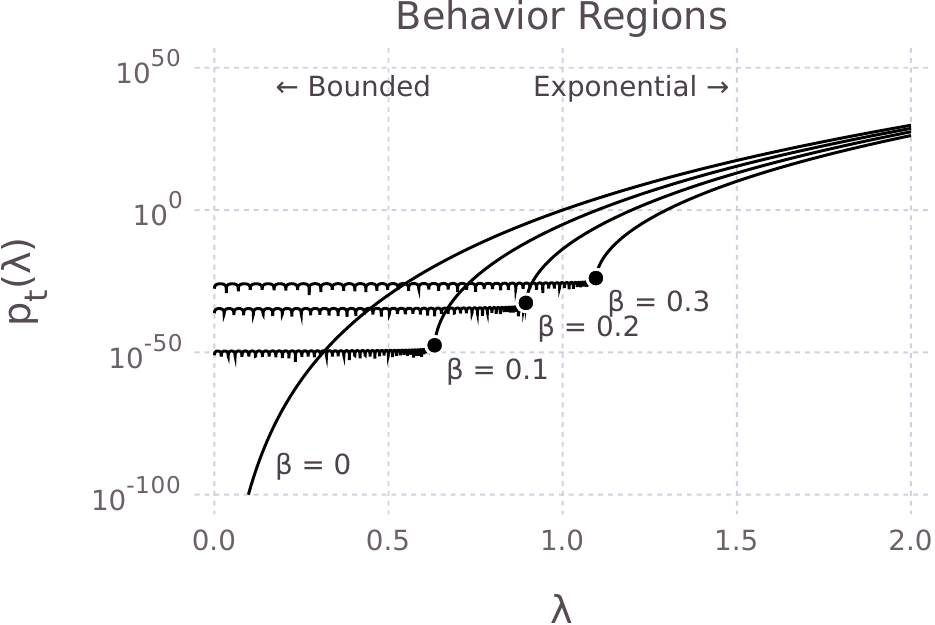}
    \caption{
      Behavior of polynomial recurrence \ref{alg: momentum} for several values of $\beta$. The recurrence is run for $t=100$ steps.
    }
    \label{fig:regions}
  \end{center}
\end{figure}

For power iteration, where $\beta = 0$, $p_t(\lambda) = \lambda^t$.
While the recurrence reduces mass on small eigenvalues quickly, eigenvalues near the largest eigenvalue will decay relatively slowly, yielding slow convergence.

As $\beta$ is increased, a ``knee'' appears in $p_t(\lambda)$.
For values of $\lambda$ smaller than the knee, $p_t(\lambda)$ remains small, which implies that these eigenvalues decay quickly.
For values of $\lambda$ greater than the knee, $p_t(\lambda)$ grows rapidly, which means that these eigenvalues will remain.
By selecting a $\beta$ value that puts that knee close to $\lambda_2$, our recurrence quickly eliminates mass on all but the largest eigenvector.

\noindent\textbf{Well-Conditioned Augmented Matrix}

Consider the recurrence
\begin{align}\label{alg: aug}
  \begin{pmatrix}
    \tilde \w_{t+1} \\ \tilde \w_t
  \end{pmatrix} = 
  \begin{pmatrix}\A  & -\beta I \\ I & 0 \end{pmatrix} \begin{pmatrix}
    \tilde \w_{t} \\ \tilde \w_{t-1}
\end{pmatrix}.  \tag{{\bf A1}}
\end{align}

Notice that this is simply power iteration on an augmented matrix.
It is straightforward to see taht the power iteration with momentum is exactly equivalent to standard power iteration on this augmented matrix, i.e. $\{\tilde\w_t\}$ from \eqref{alg: aug} and $\{\w_t\}$ from \eqref{alg: generic} are the same. 
As a result, we can take advantage of known power iteration properties when studying our method. In the following proposition, we derive the eigenvalues of the augmented matrix.
\begin{restatable}{proposition}{thmaugeig}
\label{thm: augeig}
Suppose a matrix $\A$ has eigenvalue- eigenvector pairs $(\lambda_i, \u_i)_{i=1}^n$, then the augmented matrix
  \begin{align*}
  \M = \begin{pmatrix}\A & -\beta I \\ I & 0 \end{pmatrix}
  \end{align*}
 has eigenvalue-eigenvector pairs $$\left(\frac{\lambda_i \pm \sqrt{\lambda_i^2 - 4\beta}}{2}, \begin{pmatrix}\frac{\lambda_i \pm \sqrt{\lambda_i^2 - 4\beta}}{2}\u_i\\ \u_i \end{pmatrix}\right)_{i=1}^n.$$
\end{restatable}
In particular, when $\lambda_2\le 2\sqrt\beta < \lambda_1$, the relative eigen-gap of this augmented matrix is $1 - \dfrac{2\sqrt\beta}{\lambda_1 + \sqrt{\lambda_1^2 - 4\beta}}$. And the standard power iteration on $\M$ has the convergence rate $\bigO\left(\left(\frac{2\sqrt\beta}{\lambda_1 + \sqrt{\lambda_1^2 - 4\beta}}\right)^{2t}\right)$, which matches the result in Theorem \ref{thm: acc_rate}.

Now we present the proof of Proposition \ref{thm: augeig} below.
\begin{proof}
  For any eigenvalue, eigenvector pair $(\lambda, \u)$ of $\A$, let $\mu$ be a solution of $\mu^2 - \lambda\mu + \beta = 0$.
  Suppose that we define
  \begin{align*}
      \M = \begin{pmatrix}\A & -\beta I \\ I & 0 \end{pmatrix},~~\v = \begin{pmatrix}\mu \u\\ \u\end{pmatrix}.
  \end{align*}
  Then,
  \begin{align*}
    \M \v &= \begin{pmatrix}\A & -\beta I \\ I & 0 \end{pmatrix}
          \begin{pmatrix}\mu \u\\ \u\end{pmatrix} 
        = \begin{pmatrix}\mu \A \u - \beta \u\\\mu \u\end{pmatrix} 
        = \begin{pmatrix}\lambda\mu \u - \beta \u\\\mu \u\end{pmatrix} 
        = \begin{pmatrix}\mu^2 \u\\\mu \u\end{pmatrix} 
        = \mu\begin{pmatrix}\mu \u\\ \u\end{pmatrix} = \mu \v.
  \end{align*}
  Thus, $\v$ is an eigenvector of $\M$ with corresponding eigenvalue $\mu$.
  Doing this for all eigenvectors of $\A$ will produce a complete eigendecomposition of $\M$.
\end{proof}

\section{Extensions}
\label{sec: extensions}
In this section, we consider several extension based on power method with momentum presented in Section \ref{sec: momentum_pca}. In Section \ref{subsec: block}, we will generalize our methods to multiple components case, i.e. finding the top $k$ eigenvalues/eigenvectors and show that it is numerically stable in Section \ref{subsec:stability}. In Section \ref{sec: tuning}, we provide some simple heuristics to tune the momentum parameter.
In Section \ref{subsec: inhomogeneous}, we extend our momentum method into an inhomogeous polynomials recurrence and show that it is optimal in expection with respect to the tail distribution of the tail spectrum of the target matrix $\A$. All the proofs for this section are in Section \ref{sub_sec: extend_proofs}.

\subsection{Block Update for Multiple Components}
\label{subsec: block}

In this section, we use a block version of our method to compute multiple principal components.
In this case, the initial state is a matrix $\W_0\in\bbR^{d\times k}$, rather than a single vector.
The orthogonal polynomal sequence \eqref{alg: momentum} natually corresponds to the update scheme
\begin{align}\label{alg: bpm}
\W_{t+1} = \A\W_t - \beta \W_{t-1}. \tag{{\bf A'}}
\end{align}
To obtain the convergence result, we use the standard definition from \citet{golub2012matrix} to measure the distance between spaces. 
\begin{definition}
Given two spaces $S_1, S_2 \subseteq \bbR^d$, the distance between $S_1, S_2$ is defined as 
$$\textnormal{dist}(S_1, S_2) = \|\P_1 - \P_2\|_2,$$
where $\P_i$ is the orthogonal projection onto $S_i$. Furthermore, when $S_1,S_2$ are matrices, we overload the definition as $\textnormal{dist}(S_1,S_2) =\textnormal{dist}(range(S_1),range(S_2))$, where $range(\cdot)$ denotes the range space.
\end{definition}

The following lemma shows that we can analyze the convergence rate of any update scheme by studying the growth rate of the corresponding orthogonal polynomial.
\begin{restatable}{lemma}{lembp}
\label{lem: bp}
  Given a PSD matrix $\A\in\bbR^{d\times d}$, its top $k$ ($1\le k < d$) eigenvectors $\U_k\in\bbR^{d\times k}$, and a matrix $W_0\in\bbR^{d\times k}$ such that $d_0 = \textnormal{dist}(\U_k, \W_0) \ne 1$, for any polynomial $p(\cdot)$, we have
  \begin{small}
\begin{align*}
\textnormal{dist}(p(\A)\W_0, \U_k) \le \frac{d_0}{\sqrt{1-d_0^2}}\cdot \max_{\substack{i=1,\dots,k;\\j=k+1,\dots,n}}\left|\frac{p(\lambda_j)}{p(\lambda_i)}\right |.
\end{align*}
  \end{small}\noindent 
\end{restatable}

The following theorem gives the rate at which the space spanned by the first $j$ columns of $\W_t$ approach the space spanned by the top $j$ eigenvectors.
\begin{restatable}{theorem}{thmbpmrate}
\label{thm: bpm_rate}
  Let $\W_t^{(:j)}$ denote the first $j$ columns of $\W_t$ for $1\le j \le k$.
  Given a PSD matrix $\A\in\bbR^{d\times d}$, its top $k$ ($1\le k < d$) eigenvectors $U_k\in\bbR^{d\times k}$, a matrix $\W_0\in\bbR^{d\times k}$ such that $d_0 = \textnormal{dist}(\U_k, \W_0) \ne 1$, and $\beta$ such that $2\sqrt\beta < \lambda_k$, the update scheme \eqref{alg: momentum} results in the top $j$-eigenspace converging at a rate of
  \begin{small}
\begin{align*}
&\textnormal{dist}(\W_t^{(:j)}, \U_j) \le \frac{\textnormal{dist}(\W_0^{(:j)}, \U_j)}{\sqrt{1 - \textnormal{dist}(W_0^{(:j)}, U_j)^2}}\cdot\left(\frac{\lambda_{j+1} + \sqrt{\lambda_{j+1}^2 - 4\beta}}{\lambda_j + \sqrt{\lambda_j^2 - 4\beta}}\right)^{t}, j=1,\dots,k-1
\\
&\textnormal{dist}(\W_t^{(:k)}, \U_k) \le\frac{d_0}{\sqrt{1-d_0^2}}\cdot\begin{cases}2\left(\frac{2\sqrt\beta}{\lambda_k + \sqrt{\lambda_k^2 - 4\beta}}\right)^{t}, & \lambda_{k+1}< 2\sqrt\beta \\ \left(\frac{\lambda_{k+1} + \sqrt{\lambda_{k+1}^2 - 4\beta}}{\lambda_k + \sqrt{\lambda_k^2 - 4\beta}}\right)^{t}, & \lambda_{k+1} \ge 2\sqrt\beta \end{cases}.
\end{align*}
  \end{small}\noindent 
\end{restatable}


\subsection{Stable Implementation of Momentum Methods}
\label{subsec:stability}
In this section, we provide a numerically stable implementation of our momentum method for the multi-component case.
This implementation can also be applied in the single component case.
Consider the update scheme~\ref{alg: bpm}.
Similar to the unnormalized simultaneous iteration (which essentially is the block version of the power method) (\citealp[Lecture 28]{trefethen1997numerical}),  as $t\to \infty$, all columns of $W_t$ converge to the multiples of the same dominant eigenvectors of $A$ due to the round-off errors.
A common technique to remedy the situation is orthonormalization, which is used in the standard power method.
However we cannot simply orthonormalize each $\w_t$ or $W_t$ every iteration because it changes the convergence behavior.
Instead, we propose the normalization scheme \ref{alg: nbpm} to stabilize our method:

\vspace{-0.3cm}
\begin{small}
\begin{equation}\label{alg: nbpm}
\begin{aligned}
\tilde \W_{t+\frac{1}{2}} &= \A\tilde \W_t - \beta\tilde \W_{t-1} \R_t^{-1},
\\ \tilde \W_{t+1} &= \tilde \W_{t+\frac{1}{2}} \R_{t+1}^{-1},
\end{aligned}\tag{{\bf A''}}
\end{equation}
\end{small}\noindent 
where $\R_t\in\bbR^{k\times k}$ is an invertible upper triangular matrix and $\R_1 = I$.

First, Lemma \ref{lem: nbpm} shows that $\tilde \W_t$ generated by the normalized update scheme ~\ref{alg: nbpm} is the same as $\W_t$ generated by the original update up to a invertible upper triangular matrix factor on the right side. Therefore, the column spaces of $\tilde \W_t$ and $\W_t$ are the same, so the normalized update scheme has the same convergence property as the scheme \ref{alg: bpm}.
\begin{restatable}{lemma}{lemnbpm}
\label{lem: nbpm}
Suppose $\{\W_t\}$ and $\{\tilde \W_t\}$ are the two sequences generated by \eqref{alg: bpm} and \eqref{alg: nbpm} respectively and $\W_0 = \tilde \W_0, \W_1 = \tilde \W_1$, then $\tilde \W_t = \W_t \C_t$ where $\C_t\in \bbR^{k\times k}$ is an invertible upper triangular matrix for any $t > 0$.
\end{restatable}

Now consider the actual implementation of scheme \ref{alg: nbpm}.
One choice of $\R_{t+1}$ is found by using the QR factorization 
\begin{small}
$$\begin{pmatrix}\tilde \W_{t+\frac{1}{2}}\\ \tilde \W_t\end{pmatrix} = \begin{pmatrix}\tilde \W_{t+1}\\ \tilde \W_t \R_{t+1}^{-1}\end{pmatrix} \R_{t+1}.$$
\end{small}\noindent 
In this case, the iteration \eqref{alg: nbpm} is indeed \textit{backward stable}. In fact, the update \eqref{alg: nbpm} with the choice of $\R_t$ above is equivalent to the normalized simultaneous iteration on the augmented matrix\footnote{In general the normalized simultaneous iteration converges to the Schur vectors of the matrix, not the eigenvectors because the matrix is not Hermitian. However in our particular problem, the normalized simultaneous iteration on the augmented matrix can converge to the eigenvectors of $\A$.}
$\hat \A$, which has backward stablilty \cite{golub2012matrix}.
Also notice that we do not have to materialize the augmented matrix and $\tilde \W_{t-1}\R_t^{-1}$ and $\tilde \W_{t+\frac{1}{2}} \R_{t+1}^{-1}$ is done implicitly through QR factorization.

\begin{figure}[htbp]
  \begin{center}
    \includegraphics[scale=0.5]{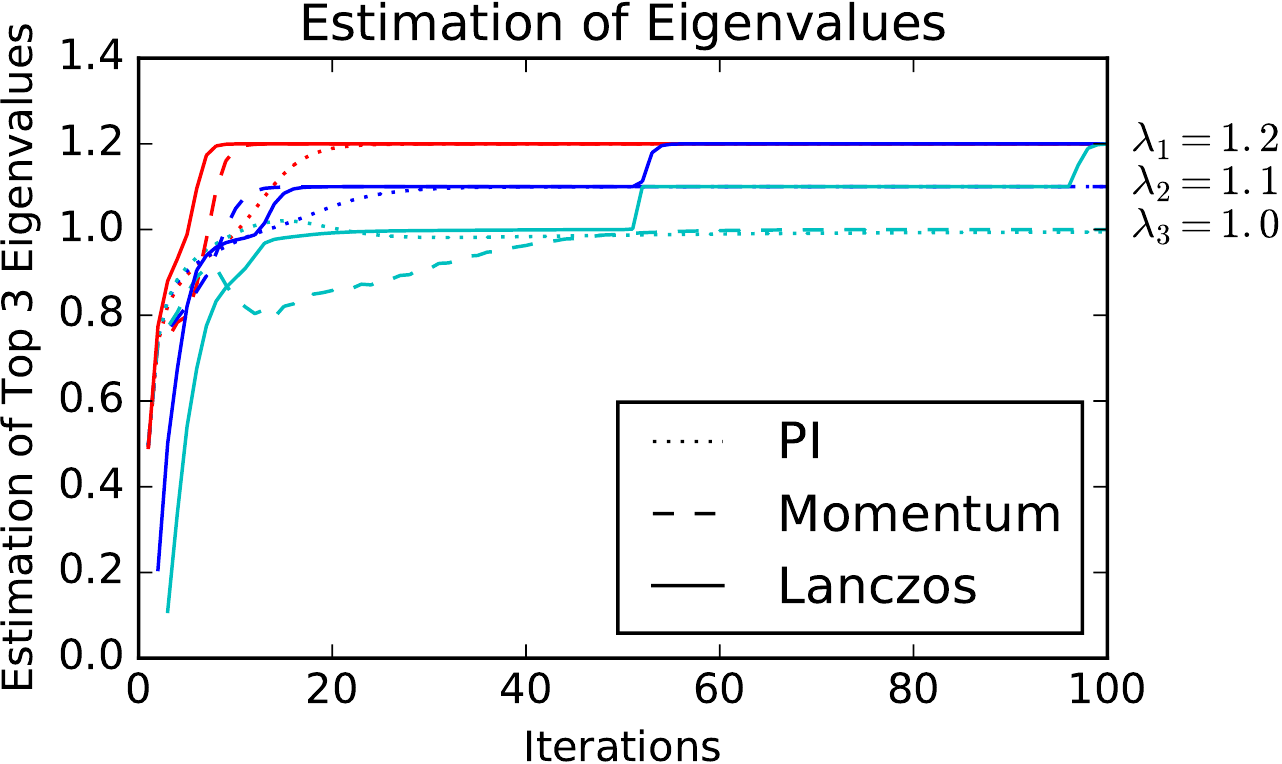}
    \vspace{-0.3cm}
    \caption{
      Convergence of standard power iteration, power iteration with momentum, and the Lanczos algorithm to the top eigenvalues of a matrix.
      Estimation of the first eigenvalue (red), second eigenvalue (blue), and third eigenvalue (cyan) are shown.
    }
  \label{fig: stability}
  \end{center}
\end{figure}

We now experimentally demonstrate the efficiency and stability of our method. In Figure~\ref{fig: stability}, we show the estimates of the top three eigenvalues produced by standard power iteration, power iteration with momentum, and the classic Lanczos method. 
First, notice that the Lanczos iteration is not numerically stable because of the ``ghost'' eigenvalues problem (\citealp[Lecture 36]{trefethen1997numerical}).
The estimates of the top three eigenvalues produced by the Lanczos algorithm eventually all converge to the top eigenvalue.
In contrast, both standard power iteration and power iteration with momentum successfully find all three eigenvalues.
However, standard power iteration takes much longer than power iteration with momentum to converge.

\subsection{Tuning Momentum}
\label{sec: tuning}

Our optimal momentum $\beta$ is determined by $\lambda_2$, which will not always be known a priori.
We introduce the best heavy ball method to automatically tune $\beta$ in real time.
\vspace{-0.2cm}
\begin{algorithm}[H]
  \caption{Best Heavy Ball}
  \begin{algorithmic}
    \REQUIRE $d\times d$ Matrix $\A$, Number of Iterations $T$
    \STATE $\w$ $\leftarrow$ Random $n$-dimensional vector
    \STATE $\mu$ $\leftarrow$ $\left(\w^T\A\w\right) / \left(\w^T\w\right)$
    \STATE $\beta$ $\leftarrow$ $\mu^2 / 4$
    \FOR {t = 1 \TO T}
      \STATE Run 10 steps with $2/3\beta, 0.99\beta, \beta, 1.01\beta, 1.5\beta$
      \STATE Set $\beta$ to momentum with largest Rayleigh quotient
    \ENDFOR
    \RETURN $\w$ that gives the largest Rayleigh quotient.
  \end{algorithmic}
\end{algorithm}
\vspace{-0.3cm}

In the heavy ball method, an arbitrary matrix is taken as input, and no information about the matrix is required.
A lower bound for the largest eigenvalue is computed by computing the Rayleigh quotient of a random initial vector.
This estimate is used to select the initial choice of $\beta$.
Afterwards, power iteration with momentum is run for 10 steps over a range of values surrounding this choice of $\beta$. 
The performance is measured by the estimation using Rayleigh quotient, i.e.,
the momentum resulting in the largest Rayleigh quotient\footnote{For the multi-component case, we take the sum of all the estimates of top $k$ eigenvalues using Rayleigh quotients. } is considered the best-performing momentum, and is used as the new center for the search.


\begin{figure}[htbp]
\centering
\subfigure[]{\includegraphics[width=0.45\textwidth]{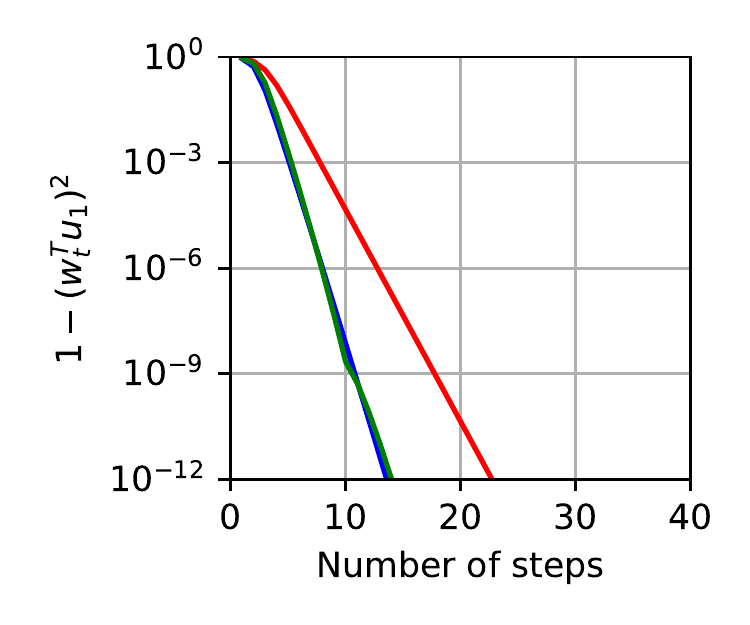}}
\subfigure[]{\includegraphics[width=0.45\textwidth]{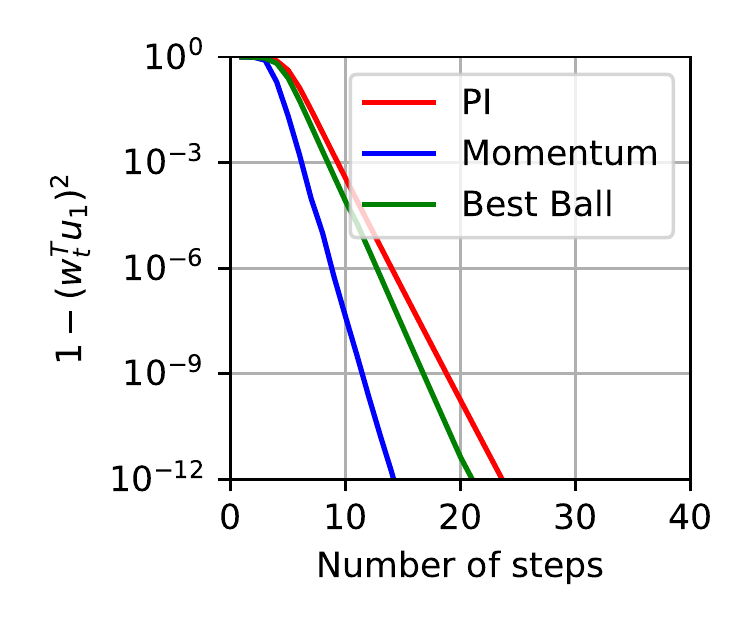}}
\subfigure[]{\includegraphics[width=0.45\textwidth]{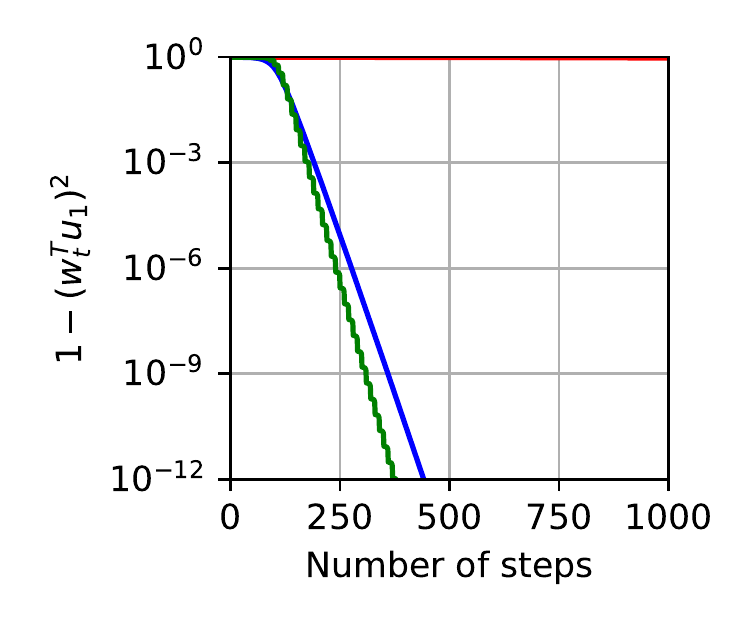}}
\subfigure[]{\includegraphics[width=0.45\textwidth]{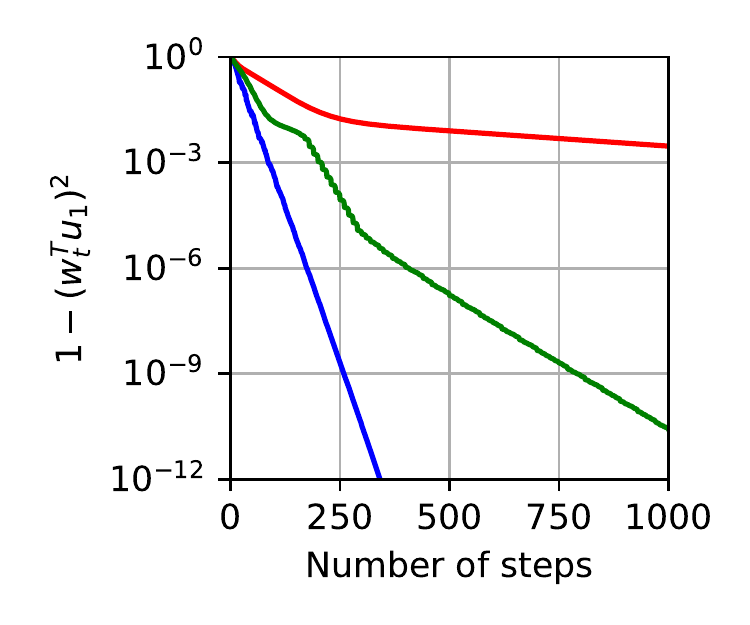}}
\caption{
      Empirical analysis of the best heavy ball method on four $1000\times1000$ matrices.
      The largest eigenvalue of all four matrices is 1.
      The remaining eigenvalues are: (a) all 0.5.
                                     (b) equally spaced from 0 to 0.5.
                                     (c) all 0.999.
                                     (d) equally spaced from 0 to 0.999.
    }
    \label{fig: bestball}
\end{figure}

Figure \ref{fig: bestball} compares the performance of power iteration, power iteration with momentum, and the best ball method.
Experiments (a) and (b) both have a large eigen-gap, so convergence is fast for all methods.
However, even though only a small number of iterations are needed, the best ball method is able to find a suitable value of $\beta$, and achieves acceleration.
Experiments (c) and (d) have a much smaller eigen-gap, so the acceleration from a well-tuned $\beta$ is critical for fast convergence.
In these experiments, we see that the best ball method is also able to select a $\beta$ that outperforms power iteration.
We also note that the best heavy ball actually outperforms power iteration with momentum in experiment (c), which suggests an inhomogeneous sequence of $\beta$'s sometimes results in superior performance.

\subsection{Inhomogeneous Polynomial Recurrence}
\label{subsec: inhomogeneous}

In this section, we present a new algorithm that goes beyond the traditional orthogonal polynomial setting of momentum methods to produce faster convergence of power iteration in some cases.
First, we will motivate and derive this method.
Suppose that we are trying to run PCA on a matrix $A$, and that, from experience with other matrices we have encountered in similar settings, we have a rough idea of the spectrum of $A$.
More concretely, suppose that we believe that the largest eigenvalue is $\lambda_1$, and the other eigenvalues are independently randomly generated according to some distribution $\mu$ (with compact support).
As in the momentum case, we want to produce a series of iterates $\w_t$ that approach the dominant eigenvector $u_1$, and can be written as
\[
  \w_t = f_t(A) \w_0
\]
where $f_t$ is a degree-$t$ polynomial analogous to $p_t$ as defined in (\ref{alg: momentum}).
Our goal is to choose some $f_t$ that can perform better than momentum method, using the extra information we have about the distribution of the spectrum.

The most straightforward way to proceed is to choose the $f_t$ that minimizes the expected error of our estimates over all degree-$t$ polynomials.
If we formulate the error of the estimate as
\[
  \epsilon_t
  =
  \frac{\| \w_t \|^2}{ (\u_1^T \w_t)^2 } - 1
  =
  \sum_{i=2}^d \frac{(\u_i^T \w_t)^2}{ (\u_1^T \w_t)^2 }
  =
  \sum_{i=2}^d \frac{f_t^2(\lambda_i) (\u_i^T \w_0)^2 }{f_t^2(\lambda_1) (\u_1^T \w_0)^2 },
\]
then
\[
  \bbE[\epsilon_t]
  =
  \frac{\bbE_{\lambda \sim \mu}[f_t^2(\lambda)]}{f_t^2(\lambda_1)}
  \sum_{i=2}^n \frac{(\u_i^T \w_0)^2 }{(\u_1^T \w_0)^2 }.
\]
It follows without loss of generality that, to minimize the error, it suffices to solve the optimization problem
\begin{align}
\label{eq: problem}
\begin{array}{ll}
\mbox{minimize} & \mathbb E_{\lambda\sim\mu}\left[f_t^2(\lambda)\right] \\
\mbox{subject to} & f_t(\lambda_1) = 1 \\
  & \textrm{$f_t$ is a degree-$t$ polynomial}.
\end{array}
\end{align}
This problem statement means that we are interested in finding a update scheme that minimizes the expected power on non-principal components, while keeping a fixed mass on the principal component. 
We can solve this problem algebraically by decomposing $f_t$ in terms of the family of polynomials $\{ q_t \}_{t=0}^{\infty}$ orthogonal with respect to the distribution $\mu$.\footnote{An orthogonal polynomial family is guaranteed to exist for any distribution with compact support~\cite{chihara2011introduction}.}

This is the unique polynomial family such that $q_t$ is degree-$t$ and
\[
  \mathbb E_{\lambda\sim\mu}\left[q_i(\lambda) q_j(\lambda)\right] = \delta_{i,j}.
\]
It turns out that we can solve Equation~\eqref{eq: problem} by representing $f_t$ as a linear combination of orthogonal polynomials from $\{ q_t \}_{t=0}^{\infty}$.
\begin{restatable}{theorem}{thmoptimal}
\label{thm: optimal}
The degree-$t$ polynomial that solves Equation~\eqref{eq: problem} is
\[
  f_t^*(\lambda)
  =
  \sum_{i=0}^t \frac{q_i(\lambda_1)}{\sum_{j=0}^t q_i^2(\lambda_1)} q_i(\lambda).
\]
\end{restatable}

Theorem \ref{thm: optimal} presents the optimal solution as a linear combination of orthogonal polynomials from a particular family.
The solution can also be written in the form
\begin{align}
\label{eq: inhomo}
  f_{n+1}^*(x) = f_n^*(x)\cdot\frac{\|\r_n\|^2}{\|\r_{n+1}\|^2} + p_{n+1}(x)\cdot\frac{p_{n+1}(\lambda_1)}{\|\r_{n+1}\|^2},
\end{align}
where $\r_n := \left[\begin{array}{ccc}q_1(\lambda_1) & \cdots & q_n(\lambda_1)\end{array}\right]$.
It turns out $f_n^*(x)$ comes from a family of polynomials which has higher-order recurrence.
Since the orthogonal polynomial $p_n(x)$ satisfies 3-term recurrence, i.e.
\begin{align}\label{eq: p_tilde}
p_{n+1}(x) = (\tilde a_n x + \tilde c_n) p_n(x) - \tilde b_n p_{n-1}(x)
\end{align}
where $\tilde a_n , \tilde b_n , \tilde c_n\in\bbR$ depend on the measure $\mu$, then \eqref{eq: inhomo} can be simplified into the follow four-term recurrence
\begin{small}
\begin{align}
\label{eq: 4term}
  f_{n+2}(x) &= (a_{n+1}x - b_{n+1}) f_{n + 1}(x) + (c_{n+1}x - d_{n+1})f_n(x) + e_{n+1}f_{n-1}(x).
\end{align}
\end{small}\noindent
And the derivation can be seen in Appendix \ref{a_sec: inhomo}. 

In \eqref{eq: inhomo}, the update scheme depends on $\lambda_1$, in practice we usually don't know the exact value of $\lambda_1$. However we can replace $\lambda_1$ with an underestimate $\tilde\lambda_1$ such that $\tilde\lambda_1\le \lambda_1$ and $\tilde\lambda_1> \lambda_2$. The actual algorithm based on the scheme \eqref{eq: inhomo} is presented in Algorithm \ref{alg: inhomo}. 

\begin{algorithm}[H]
  \caption{Inhomogeneous Recurrence Algorithm}
  \begin{algorithmic}
    \REQUIRE $d\times d$ Matrix $\A$, Number of Iterations $T$, Underestimate of $\tilde\lambda_1$ ($\lambda_2\le \tilde\lambda_1< \lambda_1$)
    \STATE Initial values: $\w_0=\p_0, \w_1=\p_1 \in\bbR^d$, $r_1, \tilde p_0, \tilde p_1\in\bbR_+$
    \FOR {t = 1 \TO T}
      \STATE $\p_{t+1} \gets (\tilde a_t \cdot \A + \tilde c_t)\p_t - \tilde b_t \p_{t-1} $
      \STATE $\tilde p_{t+1} \gets  (\tilde a_t \cdot \tilde\lambda_1 + \tilde c_t)\tilde p_t - \tilde b_t \tilde p_{t-1} $
      \STATE $r_{t+1} \gets r_{t} + \tilde p_{t+1}^2$
      \STATE $\w_{t+1} \gets \w_t \cdot  r_{t}/r_{t+1}  + \p_{t+1} \cdot  \tilde p_{t+1}/r_{t+1}$
      \STATE Normalization: 
	\begin{align*}&\w_{t+1}\gets \frac{\w_{t+1}}{\|\w_{t+1}\|}, \w_{t}\gets \frac{\w_{t}}{\|\w_{t+1}\|},\w_{t-1}\gets \frac{\w_{t-1}}{\|\w_{t+1}\|},
	\\ &\p_{t+1}\gets \frac{\p_{t+1}}{\|\w_{t+1}\|}, \p_{t}\gets \frac{\p_{t}}{\|\w_{t+1}\|},\p_{t-1}\gets \frac{\p_{t-1}}{\|\w_{t+1}\|}
	\end{align*}
    \ENDFOR
    \RETURN $\w_T$ as the estimation of the largest eigenvector.
  \end{algorithmic}
\label{alg: inhomo}
\end{algorithm}
The implementation of Algorithm \ref{alg: inhomo} is based on the equation \eqref{eq: inhomo} and \eqref{eq: p_tilde}. More concretely, ignoring the normalization procedure, we have 
$\p_t = p_t(A)\w_0$, $\w_t = f_t(A)\w_0$, $\tilde p_t = p_t(\tilde\lambda_1)$ and $r_t = \|\r_t\|^2$.

\noindent
{\it Example.} Now we give a concrete example to show the inhomogeneous algorithm works better than momentum method. 
Figure \ref{fig: inhomogeneous} shows the performance of the different update schemes on a $500\times 500$ matrix.
The principal eigenvalue is 1.001, and the remaining eigenvalues are uniformly selected from the interval $[-1,1]$.
This measure corresponds to the Legendre polynomial family.
In this example, we see that the loss of the optimal update scheme is essentially always lower than the loss of either power iteration or constant momentum.
This indicates that more complex recurrences are required for obtaining ideal performance.
\begin{figure}[htbp]
  \begin{center}
    \includegraphics[scale=0.45]{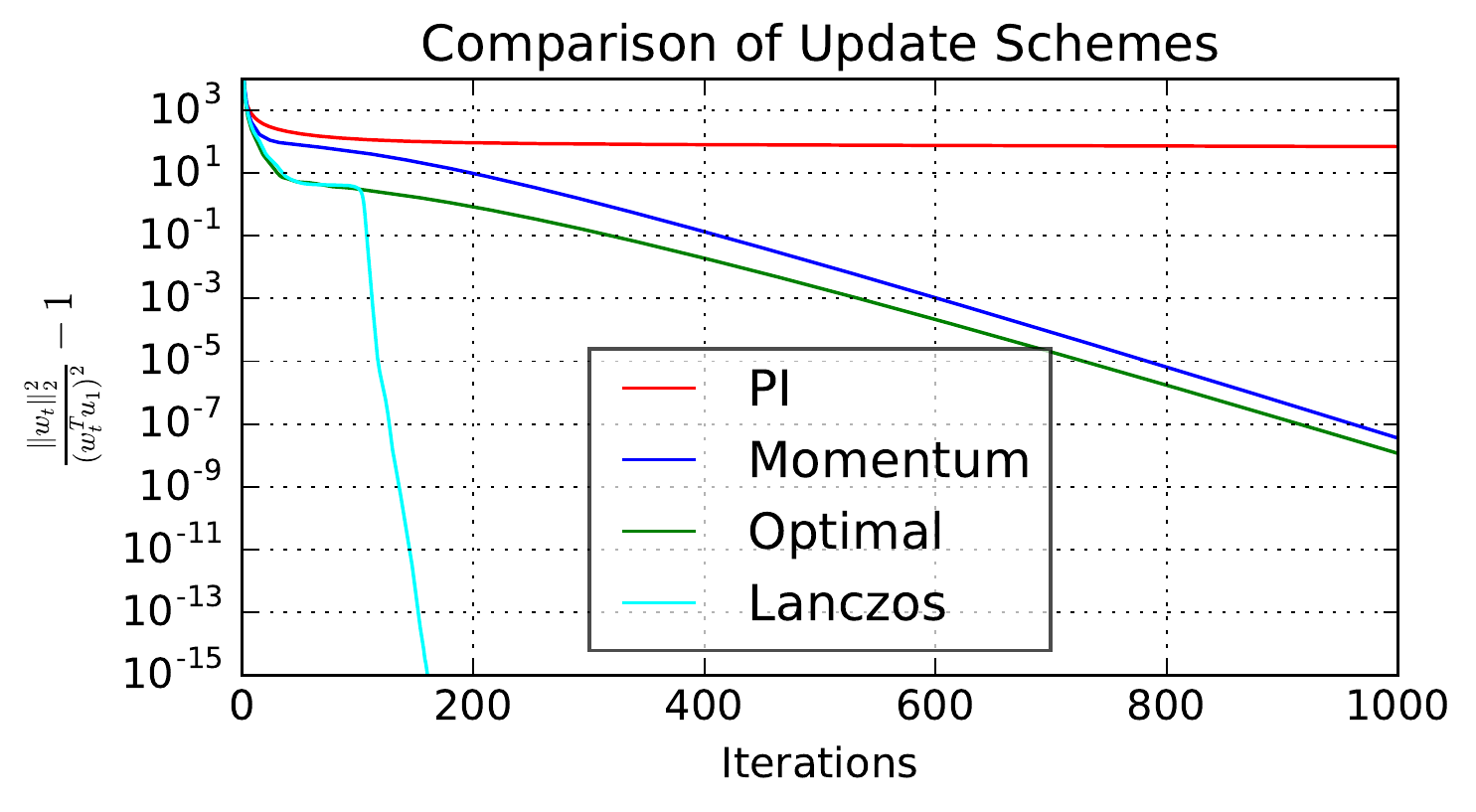}
    \vspace{-0.3cm}
    \caption{
      Example comparing the performance of power iteration, constant momentum, and the optimal update scheme.
    }
    \label{fig: inhomogeneous}
  \end{center}
\end{figure}

\subsection{Derivation of 4-term Recurrence \eqref{eq: 4term}}\label{a_sec: inhomo}
First, we restate the inhomogeneous recurrence \eqref{eq: inhomo},
\begin{align*}
f_{n+1}(x) = f_n(x)\cdot\frac{\|\r_n\|^2}{\|\r_{n+1}\|^2} + p_{n+1}(x)\cdot\frac{p_{n+1}(\lambda_1)}{\|\r_{n+1}\|^2}.
\end{align*}
Therefore we have
\begin{align*}
p_{n+1}(x) = \frac{\|\r_{n+1}\|^2}{p_{n+1}(\lambda_1)} \cdot \left(f_{n+1}(x) - f_n(x)\cdot\frac{\|\r_n\|^2}{\|\r_{n+1}\|^2} \right).
\end{align*}
We assume that the orthogonal polynomial $p_n(x)$ has the following the three-term recurrence,
\begin{align*}
p_{n+1}(x) = \tilde a_n x p_n(x) - \tilde b_n p_{n-1}(x).
\end{align*}
Now let's consider $f_{n+2}(x)$,
\begin{align*}
f_{n+2}(x) &= f_{n+1}(x) \cdot \frac{\|\r_{n+1}\|^2}{\|\r_{n+2}\|^2} + p_{n+2}(x)\cdot\frac{p_{n+2}(\lambda_1)}{\|\r_{n+2}\|^2}
\\ &= f_{n+1}(x) \cdot \frac{\|\r_{n+1}\|^2}{\|\r_{n+2}\|^2}  + \frac{p_{n+2}(\lambda_1)}{\|\r_{n+2}\|^2} \cdot \left(\tilde a_{n+1} x p_{n+1}(x) - \tilde b_{n+1}p_n(x)\right)
\\& = f_{n+1}(x) \cdot \frac{\|\r_{n+1}\|^2}{\|\r_{n+2}\|^2} + \frac{p_{n+2}(\lambda_1)}{\|\r_{n+2}\|^2} \cdot \left(\tilde a_{n+1} x \frac{\|\r_{n+1}\|^2}{p_{n+1}(\lambda_1)} \cdot \left(f_{n+1}(x) - f_n(x)\cdot\frac{\|\r_n\|^2}{\|\r_{n+1}\|^2} \right) \right) 
\\ &~~~~~ -  \frac{p_{n+2}(\lambda_1)}{\|\r_{n+2}\|^2} \cdot \tilde b_{n+1}  \frac{\|R_{n}\|^2}{p_{n}(\lambda_1)} \cdot \left(f_{n}(x) - f_{n-1}(x)\cdot\frac{\|\r_{n-1}\|^2}{\|R_{n}\|^2} \right)
\\ &= \left(\tilde a_{n+1}\frac{p_{n+2}(\lambda_1)\|\r_{n+1}\|^2}{p_{n+1}(\lambda_1)\|\r_{n+2}\|^2} x + \frac{\|\r_{n+1}\|^2}{\|\r_{n+2}\|^2} \right) f_{n+1}(x)  
\\ &~~~~+ \left( -\tilde a_{n+1}\frac{p_{n+2}(\lambda_1)}{p_{n+1}(\lambda_1)} \frac{\|\r_n\|^2}{\|\r_{n+2}\|^2} x -  \tilde b_{n+1}\frac{p_{n+2}(\lambda_1)\|\r_n\|^2}{p_n(\lambda_1)\|\r_{n+2}\|^2}\right) f_n(x)
\\ &~~~~+ \left(\tilde b_{n+1}\frac{p_{n+2}(\lambda_1)\|\r_{n-1}\|^2}{p_n(\lambda_1) \|\r_{n+2}\|^2}\right) f_n(x)
\end{align*}

Let
\begin{align*}
a_{n+1} &= \tilde a_{n+1}\frac{p_{n+2}(\lambda_1)\|\r_{n+1}\|^2}{p_{n+1}(\lambda_1)\|\r_{n+2}\|^2}
\\
b_{n+1} &= \frac{\|\r_{n+1}\|^2}{\|\r_{n+2}\|^2}
\\
c_{n+1} &=  -\tilde a_{n+1}\frac{p_{n+2}(\lambda_1)}{p_{n+1}(\lambda_1)} \frac{\|\r_n\|^2}{\|\r_{n+2}\|^2}
\\
d_{n+1} &= -\tilde b_{n+1}\frac{p_{n+2}(\lambda_1)\|\r_n\|^2}{p_n(\lambda_1)\|\r_{n+2}\|^2}
\\
e_{n+1} &= \tilde b_{n+1}\frac{p_{n+2}(\lambda_1)\|\r_{n-1}\|^2}{p_n(\lambda_1) \|\r_{n+2}\|^2}
\end{align*}
and we get the 4-term recurrence \eqref{eq: 4term}.

\subsection{Proofs}\label{sub_sec: extend_proofs}
\thmoptimal*
\begin{proof}
  First, we substitute $f_t(\lambda) = \sum_{i=1}^{t}a_{n,i}q_i(\lambda)$ into the optimization problem.
  \begin{align*}
    \begin{array}{ll}
      \textrm{minimize} & \mathbb E_{\lambda\sim\mu}\left[\left(\sum_{i=0}^{t}a_{t,i}q_i(\lambda)\right)^2\right] \\
      \textrm{subject to} & \sum_{i=0}^{t}a_{t,i}q_i(\lambda_1) = 1.
    \end{array}
  \end{align*}

  By taking advantage of the orthogonality of the $q_i(\lambda)$, we have
  \begin{align*}
    \begin{array}{ll}
      \textrm{minimize} & \mathbb E_{\lambda\sim\mu}\left[\sum_{i=0}^{t}a_{t,i}^2q_i^2(\lambda)\right] \\
      \textrm{subject to} & \sum_{i=0}^{t}a_{t,i}q_i(\lambda_1) = 1.
    \end{array}
  \end{align*}

  Then, because $q_i(\lambda)$ is normalized (i.e. $\mathbb E_{\lambda\sim\mu}\left[q_i^2(\lambda)\right] = 1$), we have
  \begin{align*}
    \begin{array}{ll}
      \textrm{minimize} & \mathbb \sum_{i=0}^{t}a_{t,i}^2 \\
      \textrm{subject to} & \sum_{i=0}^{t}a_{t,i}q_i(\lambda_1) = 1.
    \end{array}
  \end{align*}

  This is minimized by
  \begin{align*}
    a_{t,i} = \frac{p_t(\lambda_1)}{\sum_{j=0}^t q_i^2(\lambda_1)}.
  \end{align*}
\end{proof}

\lembp*
\begin{proof}
Suppose $\A = \U \Lambda \U^T$ is the eigendecomposition of $\A$.Denote $\U_k\in\bbR^{n\times k}$ as the first $k$-columns of $\U$ (i.e. the top $k$ eigenvectors of $\A$) and $\U_{-k}\in\bbR^{n\times (n-k)}$ as the last $n-j$ columns of $\U$ (i.e. the smallest $n-k$ eigenvectors of $\A$). Correspondingly, denote $\Lambda_k\in\bbR^{k\times k}$ as the top left $k\times k$ block of $\Lambda$ and $\Lambda_{-k}\in \bbR^{(n-k)\times(n-k)}$ as the bottom right $(n-k)\times(n-k)$ block of $\Lambda$.

Suppose $p(\A)\W_0 = \Q\R$ is the QR factorization of $p(\A)\W_0$ and $\W_0 = \Q_0\R_0$ is the QR factorization of $\W_0$. Then,
\begin{align*}
\Q\R &= p(\A)\W_0
\\ &= p(\A)\Q_0\R_0
\\ &= \U p(\Lambda)\U^T\Q_0\R_0
\end{align*}
Therefore we have
\begin{align*}
\U_k^T\Q\R &= p(\Lambda_k)\U_k^T\Q_0\R_0,
\\ \U_{-k}^T\Q\R &= p(\Lambda_{-k})\U_{-k}^T\Q_0\R_0.
\end{align*}
It is not difficult to see that \citep[Theorem 2.5.1, 2.5.2]{golub2012matrix}
\begin{align*}
&d_0 = \textnormal{dist}(\W_0, \U_k) = \textnormal{dist}(\Q_0, \U_k) = \|\U_{-k}^T\Q_0\|
\\ &\sigma_{\min}(\U_k^T\Q_0)^2 + \sigma_{\max}(\U_{-k}^T\Q_0)^2 = 1.
\end{align*}
Now let's compute the distance between $p(\A)\W_0$ and $\U_k$,
\begin{align*}
\textnormal{dist}(p(\A)\W_0, \U_k) &= \|\U_{-k}^T\Q\|
\\ &= \|p(\Lambda_{-k})\U_{-k}^T\Q_0\R_0 \R^{-1}\|
\\ &=  \|p(\Lambda_{-k})\U_{-k}^T\Q_0\R_0 (p(\Lambda_k)\U_k^T\Q_0\R_0)^{-1}\U_k^T\Q\|
\\ &\le \|p(\Lambda_{-k})\|_2 \|\U_{-k}^T\Q_0\|\|(p(\Lambda_k))^{-1}\| \|(\U_k^T\Q_0)^{-1}\|\|\U_k^T\Q\|
\\ &\le \frac{d_0}{\sqrt{1-d_0^2}}\cdot \max_{\substack{i=1,\dots,k;\\j=k+1,\dots,n}}\left|\frac{p(\lambda_j)}{p(\lambda_i)}\right |.
\end{align*}
\end{proof}

\thmbpmrate*
\begin{proof}
First notice that we have $\W_t = p_t(\A)\W_0$ for any $t\ge 0$. In fact, we have $\W_t^{(:j)} = p_t(\A)\W_0^{(:j)}$ for any $1\le j \le k$. Hence we can directly apply Lemma \ref{lem: bp} and get,
\begin{align*}
\textnormal{dist}(\W_t^{(:j)}, \U_j) &\le \frac{\textnormal{dist}(\W_0^{(:j)}, \U_j)}{\sqrt{1 - \textnormal{dist}(\W_0^{(:j)}, \U_j)^2}}\cdot \max_{\substack{i=1,\dots,j;\\i'=j+1,\dots,n}}\left|\frac{p_t(\lambda_{i'})}{p_t(\lambda_i)}\right |.
\end{align*}
Now since $2\sqrt{\beta} < \lambda_k$, according to Lemma \ref{lem: op},
\[
p_t(\lambda_i) = \begin{cases} \frac{1}{2} \left[\left(\frac{\lambda_i - \sqrt{\lambda_i^2 - 4\beta}}{2}\right)^t + \left(\frac{\lambda_i + \sqrt{\lambda_i^2 - 4\beta}}{2}\right)^t \right], & \lambda_i \ge 2\sqrt{\beta} 
\\ (\sqrt\beta)^t \cos\left(t\arccos(\frac{\lambda_i}{2\beta})\right) &, \lambda_i \le 2\sqrt\beta \end{cases}
\]
So, plug the polynomials in, 
\begin{align*}
\textnormal{dist}(\W_t^{(:j)}, \U_j) &\le \frac{\textnormal{dist}(\W_0^{(:j)}, \U_j)}{\sqrt{1 - \textnormal{dist}(\W_0^{(:j)}, \U_j)^2}}\cdot \left(\frac{\lambda_{j+1} + \sqrt{\lambda_{j+1}^2 - 4\beta}}{\lambda_j + \sqrt{\lambda_j^2 - 4\beta}}\right)^{t}, j=1,\dots,k-1.
\\
\textnormal{dist}(\W_t^{(:k)}, \U_k) &\le\frac{d_0}{\sqrt{1-d_0^2}}\cdot\begin{cases}2\left(\frac{2\sqrt\beta}{\lambda_k + \sqrt{\lambda_k^2 - 4\beta}}\right)^{t}, & \lambda_{k+1}< 2\sqrt\beta \\ \left(\frac{\lambda_{k+1} + \sqrt{\lambda_{k+1}^2 - 4\beta}}{\lambda_k + \sqrt{\lambda_k^2 - 4\beta}}\right)^{t}, & \lambda_{k+1} \ge 2\sqrt\beta \end{cases}.
\end{align*}

\end{proof}

\lemnbpm*
\begin{proof}
We prove that $\tilde\W_t =  \W_t\C_t$ where $\C_t =\R_t \cdot \R_{t-1}\cdots \R_0$ by induction. Base case: $\tilde \W_0 =  \W_0, \tilde \W_1 = \W_1$. Assume $\tilde \W_i = \W_i \C_i$ holds for any $i \le t$ and consider $\tilde \W_{t+1}$ and $\W_{t+1}$,
\begin{align*}
\tilde \W_{t+1}& = \left(\A \tilde \W_t - \beta \tilde \W_{t-1} \R_{t}^{-1}\right) \R_{t+1}^{-1}
\\ &=  \left(\A \W_t\C_t^{-1} - \beta \tilde \W_{t-1} \C_t^{-1} \R_{t}^{-1}\right) \R_{t+1}^{-1}
\\ &=  (\A \W_t\C_t^{-1} - \beta \W_{t-1} \C_t^{-1}) \R_{t+1}^{-1}
\\& = (\A \W_t  - \beta \W_{t-1}) \C_t^{-1} \R_{t+1}^{-1}
\\& = \W_{t+1} \C_{t+1}^{-1}.
\end{align*}
Therefore, $\W_t\C_t =  \tilde \W_t$ holds for any $t \ge 0$. 
\end{proof}
\section{Convergence Analysis for Stochastic Power methods with Momentum}\label{a_sec: stochastic_pca}
In this section we show the detailed analysis for stochastic power memthods with momentum presented in Section \ref{sec: noisy}. Here is the notation we will use for this section.
\paragraph{Notation:} $T_t(z)$ is the $t$-th degree Chebyshev polynomial of the first kind, which satisfies the recurrence,
\[
T_{t+1}(z) = 2z T_t(z) - T_{t-1}(z), T_1 = z, U_0 = 1.
\]
$U_t(z)$ is the $t$-th degree Chebyshev polynomial of the second kind, which satisfies the recurrence,
\[
U_{t+1}(z) = 2z U_t(z) - U_{t-1}(z), U_1 = 2z, U_0 = 1.
\]
$p_t(z)$ is the $t$-th degree orthogonal polynomial which satisfies the recurrence,
\[
p_{t+1}(z) = z p_t(z) - \beta p_{t-1}(z), p_1= z, p_{0} = 1.
\]
$S_m^n$ denotes the set of vectors in $\bbN^n$ with entries that sum to $m$, i.e.
\[
S_m^n = \{ \k = (k_1,\cdots, k_n)\in\bbN^n | \sum_{i=1}^n k_i = m\}.
\]
$\otimes$ denotes the Kronecker product. 

\subsection{Convergence analysis for Algorithm \ref{alg: mini_power_m}}\label{a_subsec: mini_pca}

Consider the following stochastic matrix sequence $\{\F_t\}$, which satisfies $\F_0 = I, F_{-1} = \mathbf 0$, and 
\begin{align}
\F_{t+1} = \A_{t+1} \F_t - \beta \F_{t-1}, \forall t \ge 0.
\end{align}
Here $\A_t\in\bbR^{d\times d}$ is i.i.d. stochastic matrix,  with $\Expect{\A_t} = \A$ and $\Expect{(\A_t - \A)\otimes(\A_t - \A)} = \Sigma$. 

\lemnonsvrgcov*
\begin{proof}
First, let 
\[
  \M_t = \left[\begin{array}{c c} \A_t & -\beta I \\ I & 0 \end{array}\right], \M = \left[\begin{array}{c c} \A & -\beta I \\ I & 0 \end{array}\right]
\]
and 
\[
  \E_1 = \left[\begin{array}{c} I \\ 0 \end{array}\right].
\]
and then we have
\[
\F_t = \E_1^T \cdot \M_t\cdots\M_1\cdot \E_1.
\]
Therefore we have the second moment,
\begin{align*}
  \Expect{\F_t \otimes \F_t}
  &=
  \Expect{\left(
    \E_1^T
    \cdot \M_t \cdot \M_{t-1} \cdots \M_1 \cdot
    \E_1
  \right)^2}
  \\ &=
  \Expect{\left(
    \E_1^T
    \cdot \M_t \cdot \M_{t-1} \cdots \M_1 \cdot
    \E_1
  \right)^{\otimes 2}}
  \\ &=
  \Expect{
    (\E_1^T)^{\otimes 2}
    \cdot \M_t^{\otimes 2} \cdot\M_{t-1}^{\otimes 2} \cdots \M_1^{\otimes 2} \cdot
    \E_1^{\otimes 2}
  }
  \\ &=
  (\E_1 \otimes \E_1)^T
  \cdot \Expect{\M_t^{\otimes 2}} \cdot \Expect{\M_{t-1}^{\otimes 2}} \cdots \Expect{\M_1^{\otimes 2}} \cdot
  (\E_1 \otimes \E_1)
\end{align*}
Since the $\M_i$ are i.i.d. as before, all the expected values in the last expression above will be the same.
\begin{align*}
  \Expect{\M_t^{\otimes 2}}
  &=
  \Expect{
    \left[\begin{array}{c c} \A_t & -\beta \\ I & 0 \end{array}\right]^{\otimes 2}
  }
  \\ &=
  \Expect{
    \left(
      \left[\begin{array}{c c} \A & -\beta \\ I & 0 \end{array}\right]
      +
      \left[\begin{array}{c c} \A_t - \A & 0 \\ 0 & 0 \end{array}\right]
    \right)^{\otimes 2}
  }
  \\ &=
  \Expect{
    \left(
      \M
      +
      \E_1 (\A_t - \A) \E_1^T
    \right)^{\otimes 2}
  }
  \\ &=
  \M \otimes \M
  +
  (\E_1 \otimes \E_1)
  \Expect{(\A_t - \A) \otimes (\A_t - \A)}
  (\E_1 \otimes \E_1)^T
  \\ &=
  \M \otimes \M
  +
  (\E_1 \otimes \E_1) \Sigma (\E_1 \otimes \E_1)^T.
\end{align*}
Therefore,
\begin{align*}
  \Expect{\F_t \otimes \F_t}
  &=
  (\E_1 \otimes \E_1)^T
  \left(
    \M \otimes \M
    +
    (\E_1 \otimes \E_1) \Sigma (\E_1 \otimes \E_1)^T
  \right)^t
  (\E_1 \otimes \E_1)
  \\ &=
  \sum_{n = 0}^t
  \sum_{\k \in S_{t - n}^{n + 1}}\prod_{i=2}^{n+1}
  \left(p_{k_{i}}^{\otimes 2}(\A; \beta)
  \cdot
  \Sigma\right)
  \cdot
  p_{k_1}^{\otimes 2}(\A; \beta).
\end{align*}
The last equality follows from the binomial expansion of matrices (Fact \ref{fact}).

And further we have 
\begin{align*}
  \Expect{\F_t \otimes \F_t} - \Expect{\F_t} \otimes \Expect{\F_t}
  &=
  \sum_{n = 1}^t
  \sum_{\k \in S_{t - n}^{n + 1}}
 \prod_{i=2}^{n+1}
  \left(p_{k_{i}}^{\otimes 2}(\A; \beta)
  \cdot
  \Sigma\right)
  \cdot
  p_{k_1}^{\otimes 2}(\A; \beta).
\end{align*}
Taking the norm, and knowing $0 \preceq \A \preceq \lambda_1 I$,
\begin{align*}
  \norm{ \Expect{\F_t \otimes \F_t} - \Expect{\F_t} \otimes \Expect{\F_t} }
  &\le
  \sum_{n = 1}^t
  \sum_{\k \in S_{t - n}^{n + 1}}
 \prod_{i=2}^{n+1}
  \left(\norm{p_{k_{i}}^{\otimes 2}(\A; \beta)}
  \cdot
  \norm{\Sigma}\right)
  \cdot
  \norm{p_{k_1}^{\otimes 2}(\A; \beta)}
  \\ &=
  \sum_{n = 1}^t
  \norm{\Sigma}^n
  \sum_{\k \in S_{t - n}^{n + 1}}
  \prod_{i=1}^{n+1}
  \norm{ p_{k_i}(\A; \beta) }^2
  \\ &\le
  \sum_{n = 1}^t
  \norm{\Sigma}^n
  \sum_{K \in S_{t - n}^{n + 1}}
  \prod_{i=1}^{n+1}
  p_{k_i}^2(\lambda_1; \beta)
  \\ &=
  \sum_{n = 1}^t
  \norm{\Sigma}^n
  \beta^{t - n}
  \sum_{\k \in S_{t - n}^{n + 1}}
  \prod_{i=1}^{n+1}
  U_{k_i}^2\left(\frac{\lambda_1}{2 \sqrt{\beta}} \right).
\end{align*}
The last equality follows from the fact that $p_t(x) = (\sqrt\beta)^{t}\cdot U_t(\frac{x}{2\sqrt\beta})$.
This is what we wanted to show.
\end{proof}
\noindent
{\bf Remark.} It is straightforward to see that this analysis can be applied to the case $\beta = 0$ which is the power iteration case. 

\begin{restatable}{corollary}{corollarycovbound}\label{cor: covbound}
Under the same condidtion in Lemma \ref{lem: covbound}, we have 
\[
    \norm{ \Expect{\F_t \otimes \F_t} - \Expect{\F_t} \otimes \Expect{\F_t} }
    \le
    p_t^2\left( \lambda_1; \beta \right)
    \left( \exp\left( 
      \frac{4 \norm{\Sigma} t }{\lambda_1^2 - 4 \beta}
    \right) - 1 \right).
  \]
Further if 
\begin{align}
\norm{\Sigma} \le \frac{\lambda_1^2 - 4\beta}{4t},
\end{align}
we have 
\[
    \norm{ \Expect{\F_t \otimes \F_t} - \Expect{\F_t} \otimes \Expect{\F_t} }
    \le
    p_t^2\left( \lambda_1; \beta \right) \cdot
      \frac{8 \norm{\Sigma} t }{\lambda_1^2 - 4 \beta}.
  \]
\end{restatable}
\begin{proof}
First, according to Lemma \ref{lem: covbound} and \ref{lem: pcp}, we have
\begin{align*}
\norm{ \Expect{\F_t \otimes \F_t} - \Expect{\F_t} \otimes \Expect{\F_t} }
  &\le
  \sum_{n = 1}^t
  \norm{\Sigma}^n
  \beta^{t - n}
  \sum_{\k \in S_{t - n}^{n + 1}}
  \prod_{i=1}^{n+1}
  U_{k_i}^2\left(\frac{\lambda_1}{2 \sqrt{\beta}} \right)
  \\
  & \le 
  \sum_{n = 1}^t
  \norm{\Sigma}^n
  \beta^{t - n}
  \sum_{\k \in S_{t - n}^{n + 1}}
  U_{\sum_{i=1}^{n+1}k_i+n}^2\left(\frac{\lambda_1}{2 \sqrt{\beta}}\right) \cdot\frac{1}{\left(\left(\frac{\lambda_1^2}{4\beta}\right)-1\right)^n} 
  \\
  & =
  \beta^t U_t^2\left(\frac{\lambda_1}{2\sqrt\beta}\right)
  \sum_{n=1}^t
  {t \choose t-n}\frac{4^n\norm{\Sigma}^n}{(\lambda_1^2 - 4\beta)^n}
  \\
  & = p_t^2(\lambda_1) \cdot \left(\left(\frac{4\norm{\Sigma}}{\lambda_1^2 - 4\beta} + 1\right)^t - 1\right)
  \\
  & \le p_t^2(\lambda_1) \cdot \left(\exp\left(\frac{4\norm{\Sigma}t}{\lambda_1^2 - 4\beta}\right) -1 \right)
\end{align*}
If $\norm{\Sigma} \le \frac{\lambda_1^2 - 4\beta}{4t}$, by the fact that $e^x \le 1 + 2x $ for any $x\in(0,1)$, then we have
\[
    \norm{ \Expect{\F_t \otimes \F_t} - \Expect{\F_t} \otimes \Expect{\F_t} }
    \le
    p_t^2\left( \lambda_1; \beta \right) \cdot
      \frac{8 \norm{\Sigma} t }{\lambda_1^2 - 4 \beta}.
  \]
which is the desired result.
\end{proof}

\begin{restatable}{corollary}{corveccov}\label{cor: veccov}
For any $\w_0 \in\bbR^d$ such that $\norm{\w_0} = 1$, we have
\[
\norm{\Expect{\F_t\w_0\otimes \F_t\w_0} - \Expect{\F_t\w_0}\otimes \Expect{\F_t\w_0}} \le p_t^2(\lambda_1; \beta)\cdot \frac{8\norm{\Sigma} t}{\lambda_1^2 - 4\beta}.
\]
\end{restatable}
\begin{proof}
Using the mixed-product property of Kronecker product, we have
\begin{align*}
\norm{\Expect{\F_t\w_0\otimes \F_t\w_0} - \Expect{\F_t\w_0}\otimes \Expect{\F_t\w_0}} &= \norm{\Expect{\F_t\otimes\F_t}\cdot (\w_0\otimes\w_0) - (\Expect{\F_t}\otimes\Expect{\F_t}) \cdot (\w_0\otimes\w_0)}
\\ 
&\le  \norm{\Expect{\F_t\otimes\F_t} - \Expect{\F_t}\otimes\Expect{\F_t}} \cdot \norm{\w_0\otimes\w_0}
\\
& =  \norm{\Expect{\F_t\otimes\F_t} - \Expect{\F_t}\otimes\Expect{\F_t}}\cdot \norm{\w_0}^2
\\
& \le p_t^2(\lambda_1; \beta)\cdot \frac{8\norm{\Sigma} t}{\lambda_1^2 - 4\beta}.
\end{align*}

\end{proof}
\begin{restatable}{corollary}{cornumcov}\label{cor: numcov}
For any $\u,\w_0\in\bbR^d$ such that $\norm{\u} = 1,\norm{\w_0} = 1$, we have
\[
\Var{\u^T\F_t\w_0} \le p_t^2(\lambda_1;\beta) \cdot
      \frac{8 \norm{\Sigma} t }{\lambda_1^2 - 4 \beta}.
\]
\end{restatable}
\begin{proof}
\begin{align*}
\Var{\u^T\F_t\w_0} &= \Expect{(\u^T\F_t\w_0)^2} - (\Expect{\u^T\F_t\w_0})^2
\\
&= \Expect{(\u^T\F_t\w_0)\otimes (\u^T\F_t\w_0)} - \Expect{(\u^T\F_t\w_0)}\otimes\Expect{(\u^T\F_t\w_0)}
\\
&= (\u\otimes\u)^T\cdot \left(\Expect{\F_t\w_0\otimes \F_t\w_0} - \Expect{\F_t\w_0}\otimes \Expect{\F_t\w_0}\right)
\\
&\le \norm{\u\otimes\u} \cdot \norm{\Expect{\F_t\w_0\otimes \F_t\w_0} - \Expect{\F_t\w_0}\otimes \Expect{\F_t\w_0}}
\\
& \le p_t^2(\lambda_1;\beta) \cdot \frac{8 \norm{\Sigma} t }{\lambda_1^2 - 4 \beta}.
\end{align*}
The last inequality follows from Corollary \ref{cor: veccov}.
\end{proof}

\begin{restatable}{corollary}{corupperbound}\label{cor: upperbound}
Suppose $\u_2,\cdots, \u_d$ are the last $d-1$ eigenvectors of $\A$ and $2\sqrt\beta\in[\lambda_2, \lambda_1)$.
For any fixed $\w_0\in\bbR^d$ such that $\norm{\w_0} = 1$, $\delta\in(0,1)$, with probability at least $1 - \delta$, we have 
\[
\sum_{i=2}^d (\u_i^T\F_t\w_0)^2 \le p_t^2(\lambda_1;\beta)\cdot \left(\frac{8\sqrt d \norm{\Sigma} t}{\delta(\lambda_1^2 - 4\beta)} + \frac{p_t^2(2\sqrt\beta;\beta)}{\delta p_t^2(\lambda_1;\beta)}\right)
\] 
\end{restatable}

\begin{proof}
First, we consider the second momentum
\begin{align*}
\Expect{\sum_{i=2}^d(\u_i^T\F_t\w_0)^2} &= \sum_{i=2}^d\Expect{(\u_i^T\F_t\w_0)^2}
\\ &= \sum_{i=2}^d \left[ \Expect{(\u_i\F_t\w_0)^{\otimes 2}} - \Expect{\u_i^T\F_t\w_0}^{\otimes 2} \right] + \sum_{i=2}^d \Expect{\u_i\F_t\w_0}^2
\\
& = (\sum_{i=2}^d \u_i^{\otimes 2})^T\cdot \left(\Expect{\F_t^{\otimes 2}} - \Expect{\F_t}^{\otimes 2}\right) \cdot \w_0^{\otimes 2} + \sum_{i=2}^d p_t^2(\lambda_i;\beta)(\u_i^T\w_0)^2
\\
& \le \norm{\sum_{i=2}^d \u_i^{\otimes 2}} \cdot \norm{\Expect{\F_t^{\otimes 2}} - \Expect{\F_t}^{\otimes 2}} \cdot \norm{\w_0^{\otimes 2}} + p_t^2(2\sqrt\beta; \beta)
\\
& \le \sqrt d \cdot p_t^2(\lambda_1;\beta)\frac{8\norm{\Sigma} t}{\lambda_1^2 - 4\beta} + p_t^2(2\sqrt\beta;\beta)
\\
& = p_t^2(\lambda_1;\beta) \cdot \left(\frac{8\sqrt d \norm{\Sigma} t}{(\lambda_1^2 - 4\beta)} + \frac{p_t^2(2\sqrt\beta;\beta)}{p_t^2(\lambda_1;\beta)}\right)
\end{align*}
The last inequality follows from the fact $\norm{\sum_{i=2}^d \u_i^{\otimes 2}} = \sqrt{d - 1}$.
For any $\delta \in(0,1)$, by Markov's inequality we can get the desired result. 
\end{proof}

\thmnonsvrg*
\begin{proof}
In order to bound $1 - (\u_1^T\w_t)^2$, it is equivalent to bound $1 - \frac{(\u_1^T\F_t\w_0)^2}{\norm{\F_t\w_0}^2}$.
\begin{align*}
1 - \frac{(\u_1^T\F_t\w_0)^2}{\norm{\F_t\w_0}^2} &\le \frac{\sum_{i=2}^d (\u_i^T\F_t\w_0)^2}{(\u_1^T\F_t\w_0)^2}.
\end{align*}
Notice that
\[
\Expect{\u_i^T\F_t\w_0} = p_t(\lambda_i;\beta)\u_i^T\w_0.
\]
According to Corollary \ref{cor: numcov}, by Chebyshev's inequality, for any $\delta\in(0,1)$,we have
\[
  \Prob{
    \Abs{ \u_1^T \F_t \w_0 - p_t(\lambda_1; \beta) \u_1^T \w_0 }
    \ge
    \frac{1}{\sqrt{\delta}}
    \cdot
    p_t\left( \lambda_1; \beta \right)
    \cdot
    \sqrt{ \frac{8 \norm{\Sigma} t }{\lambda_1^2 - 4 \beta} }
  }
  \le
  \delta.
\]
That is,
\[
  \Prob{
    \Abs {\u_1^T \F_t \w_0}
    \le
    p_t(\lambda_1; \beta) \left(
      \Abs{\u_1^T \w_0}
      -
      \sqrt{\frac{8\norm{\Sigma}t}{(\lambda_1^2 - 4 \beta)\delta}}
    \right)
  }
  \le
  \delta.
\]
On the other hand,according to Corollary \ref{cor: upperbound},
\[
  \Prob{
    \sum_{i=2}^d (\u_i^T \F_t \w_0)^2
    \ge
    p_t^2\left( \lambda_1; \beta \right) \left(
      \frac{8 \sqrt d\norm{\Sigma}t }{(\lambda_1^2 - 4 \beta) \delta}
      +
      \frac{ p_t^2(2 \sqrt{\beta}; \beta) }{ \delta p_t^2\left( \lambda_1; \beta \right) }
    \right)
  }
  \le
  \delta.
\]
It follows by a union bound that
\[
  \Prob{
    \frac{
      \sum_{i=2}^d (\u_i^T \F_t \w_0)^2
    }{
      (\u_1^T \F_t \w_0)^2
    }
    \ge
    \left(
      \frac{8\sqrt d \norm{\Sigma} t }{(\lambda_1^2 - 4 \beta) \delta}
      +
      \frac{ p_t^2(2 \sqrt{\beta}; \beta) }{ \delta p_t^2\left( \lambda_1; \beta \right) }
    \right)
    \left(
      \Abs{\u_1^T \w_0}
      -
      \sqrt{\frac{8 \norm{\Sigma} t }{(\lambda_1^2 - 4 \beta)\delta}}
    \right)^{-2}
  }
  \le
  2 \delta.
\]
For any $\epsilon \in (0,1/16)$, when 
\[
t = \frac{\sqrt\beta}{\sqrt{\lambda_1^2 - 4\beta}}\log(\frac{1}{\delta\epsilon}),
\]
we have 
\[
\frac{p_t^2(2\sqrt\beta;\beta)}{\delta p_t^2(\lambda_1;\beta)} \le \epsilon.
\]
When
\[
\norm{\Sigma} \le \frac{(\lambda_1^2 - 4\beta)\delta\epsilon}{8\sqrt d t} =  \frac{(\lambda_1^2 - 4\beta)^{3/2}\delta\epsilon}{8\sqrt d\sqrt\beta}\log^{-1}\left(\frac{1}{\delta\epsilon}\right),
\]
we have
\[
\frac{8\sqrt{d}\norm{\Sigma} t}{\delta(\lambda_1^2 - 4\beta)} \le \epsilon.
\]
With both conditions, we have with probability at least $1-2\delta$, 
\begin{align*}
\frac{
      \sum_{i=2}^d (\u_i^T \F_t \w_0)^2
    }{
      (\u_1^T \F_t \w_0)^2
    }
    & \le \frac{2\epsilon}{\left(\Abs{\u_1\w_0} - \ \sqrt{\frac{\epsilon}{\sqrt d}}\right)^2} \le 32\epsilon.
\end{align*}
Rescale $\epsilon$ down by a factor of $32$ would lead to the desired result.
\end{proof}

To prove Corollary \ref{cor: nonsvrg}, we can simply use the fact that
\[
\norm{\Sigma} =\norm{\Expect{(\A_t - \A)^{\otimes 2}}} \le \Expect{\norm{(\A_t - \A)^{\otimes 2}}} = \Expect{\norm{\A_t-\A}^2} = \frac{\sigma^2}{s},
\]
and we immediately get a sufficent condition of batch size to satisfy the variance condition \eqref{cond: var} in Theorem \ref{thm: nonsvrg}, and that is
\[
s \ge \frac{256\sqrt d \sigma^2 T}{(\lambda_1^2 - 4\beta)\delta\epsilon} = \frac{256\sqrt d \sqrt\beta \sigma^2}{(\lambda_1^2 - 4\beta)^{3/2}\delta \epsilon} \log\left(\frac{32}{\delta\epsilon}\right).
\]
So with Theorem \ref{thm: nonsvrg}, we get the result of Corollary \ref{cor: nonsvrg}.

\subsection{Convergence analysis for Algorithm \ref{alg: vr_power_m}}\label{a_subsec: vrpca}
For the convergence analysis for Algorithm \ref{alg: vr_power_m}, we first analyze the convergence for one epoch. For that, Consider the following stochastic matrix sequence $\{\F_t\}$, which satisfies $\F_0 = I, \F_{-1} = \mathbf 0$, and 
\begin{align}
\F_{t+1} = [\A + (\A_{t+1}-\A)(I - \w_0\w_0^T)] \F_t - \beta \F_{t-1}, \forall t \ge 0.
\end{align}
Here $\A_t\in\bbR^{d\times d}$ is i.i.d. stochastic matrix,  with $\Expect{\A_t} = \A$ and $\Expect{(\A_t - \A)\otimes(\A_t - \A)} = \Sigma$. And $\w_0\in\bbR^d$ is a fixed unit vector. 

\lemsvrgvar*
\begin{proof}
First, let 
\[
  \M_t = \left[\begin{array}{c c} \A+(\A_t - \A)(I - \w_0\w_0^T) & -\beta I \\ I & 0 \end{array}\right], \M = \left[\begin{array}{c c} \A & -\beta I \\ I & 0 \end{array}\right]
\]
and 
\[
  \E_1 = \left[\begin{array}{c} I \\ 0 \end{array}\right].
\]
and then we have
\[
\F_t = \E_1^T \cdot \M_t\cdots\M_1\cdot \E_1.
\]
Therefore we have the second moment,

\begin{align*}
  \Expect{\F_t \otimes \F_t}
  &=
  (\E_1 \otimes \E_1)^T
  \cdot \Expect{\M_t^{\otimes 2}} \cdot \Expect{\M_{t-1}^{\otimes 2}} \cdots \Expect{\M_1^{\otimes 2}} \cdot
  (\E_1 \otimes \E_1)
\end{align*}
Since the $A_i$ are i.i.d. as before, all the expected values in the last expression above will be the same.
\begin{align*}
  \Expect{\M_t^{\otimes 2}}
  &=
  \Expect{
    \left[\begin{array}{c c} \A + (\A_t-\A)(I - \w_0\w_0^T) & -\beta \\ I & 0 \end{array}\right]^{\otimes 2}
  }
  \\ &=
  \Expect{
    \left(
      \left[\begin{array}{c c} \A & -\beta \\ I & 0 \end{array}\right]
      +
      \left[\begin{array}{c c} (\A_t - \A)(I - \w_0\w_0^T)  & 0 \\ 0 & 0 \end{array}\right]
    \right)^{\otimes 2}
  }
  \\&=
  \Expect{
    \left(
      \M
      +
      \E_1 (\A_t - \A)(I - \w_0\w_0^T) \E_1^T
    \right)^{\otimes 2}
  }
  \\ &=
  \M \otimes \M
  +
  (\E_1 \otimes \E_1)
  \Expect{(\A_t - \A) \otimes (\A_t - \A)}(I - \w_0\w_0^T)^{\otimes 2}
  (\E_1 \otimes \E_1)^T
  \\ &=
  \M \otimes \M
  +
  (\E_1 \otimes \E_1) \Sigma (I - \w_0\w_0^T)^{\otimes 2}(\E_1 \otimes \E_1)^T.
  \\ &= 
  \M \otimes \M + (\E_1 \otimes \E_1) \hat\Sigma (\E_1 \otimes \E_1)^T.
\end{align*}
where $\hat\Sigma = \Sigma (I - \w_0\w_0^T)^{\otimes 2}$.
Therefore,
\begin{align*}
  \Expect{\F_t \otimes \F_t}
  &=
  (\E_1 \otimes \E_1)^T
  \left(
    \M \otimes \M
    +
    (\E_1 \otimes \E_1) \hat\Sigma (\E_1 \otimes \E_1)^T
  \right)^t
  (\E_1 \otimes \E_1)
  \\ &=
  \sum_{n = 0}^t
  \sum_{\k \in S_{t - n}^{n + 1}}\prod_{i=1}^{n}
  (p_{k_{i}}^{\otimes 2}(\A; \beta)
  \cdot
  \hat\Sigma)
  \cdot
  p_{k_{n+1}}^{\otimes 2}(\A; \beta), 
\end{align*}
and 
\begin{align*}
  &\Expect{\F_t \w_0 \otimes \F_t \w_0} - \Expect{\F_t \w_0}\otimes \Expect{\F_t \w_0}\\
  &=
  \sum_{n = 1}^t
  \sum_{\k \in S_{t - n}^{n + 1}}\prod_{i=1}^n
 ( p_{k_{i}}^{\otimes 2}(\A; \beta)
  \cdot
  \hat\Sigma)
  \cdot
  p_{k_{n+1}}^{\otimes 2}(\A; \beta) \w_0^{\otimes 2}.
\end{align*}
Taking the norm, if $0 \preceq X \preceq \lambda_1 I$,
\begin{align*}
  &\norm{ \Expect{\F_t\w_0 \otimes \F_t\w_0} - \Expect{\F_t\w_0} \otimes \Expect{\F_t\w_0} }
   \\&\le
  \sum_{n = 1}^t
  \sum_{\k \in S_{t - n}^{n + 1}}
  \norm{ p_{k_{i}}^{\otimes 2}(\A; \beta) }
  \cdot
  \norm{ \hat\Sigma }
  \cdot
  \norm{ p_{k_2}^{\otimes 2}(\A; \beta) }
  \cdot
  \norm{ \hat\Sigma }
  \cdots
  \norm{ \Sigma }
  \cdot
  \norm{ (I - \w_0\w_0^T)^{\otimes 2}p_{k_{n+1}}^{\otimes 2}(\A; \beta) \w_0^{\otimes 2}}
  \\&\le
  \sum_{n = 1}^t
  \norm{\Sigma}^n
  \sum_{\k \in S_{t - n}^{n + 1}}
  \prod_{i=1}^{n}
  \norm{ p_{k_i}(\A; \beta) }^2 \norm{(I - \w_0\w_0^T)^{\otimes 2}p^{\otimes 2}_{k_{n+1}}(\A;\beta) \w_0^{\otimes 2} }
  \\ &=
  \sum_{n = 1}^t
  \norm{\Sigma}^n
  \sum_{K \in S_{t - n}^{n + 1}}
  \prod_{i=1}^{n}
  p_{k_i}^2(\lambda_1; \beta) \norm{(I - \w_0\w_0^T)p_{k_{n+1}}(\A;\beta) z}^2
  \\ &\le 4  \theta 
  \sum_{n = 1}^t
  \norm{\Sigma}^n
  \beta^{t - n}
  \sum_{\k \in S_{t - n}^{n + 1}}
  \prod_{i=1}^{n+1}
  U_{K_i}^2\left(\frac{\lambda_1}{2 \sqrt{\beta}} \right).
\end{align*}
The last inequality follows from
\begin{align*}
\norm{(I - \w_0\w_0^T)p(\A;\beta) \w_0}^2 & \le 2 \norm{(I - \w_0\w_0^T)(I -\u_1\u_1^T)p(\A;\beta)\w_0}^2 + 2 \norm{(I -\w_0\w_0^T)\u_1\u_1^Tp(\A;\beta) \w_0}^2
\\ &\le 2\|p(\A;\beta)\|^2 \norm{(I - \u_1\u_1^T)\w_0}^2 + 2 \norm{(I - \w_0\w_0^T)\u_1}^2 \norm{p(\A;\beta)}^2(\u_1^T\w_0)^2
\\ & \le 4\theta \|p(\A;\beta)\|^2,
\end{align*}
where the last inequality uses Lemma \ref{lem: ang}.
%

This is what we wanted to show.
\end{proof}
\noindent
{\bf Remark.} Comparing to Corollary \ref{cor: veccov}, which is for the non-SVRG setting, Corollary \ref{lem: svrgvar} shows the covariance is controlled by the angle bewteen $\u_1$ and $\w_0$ which leads to shrinking variance across epochs.

\begin{restatable}{corollary}{varconcentration}\label{cor: svrgveccov}
Under the same condition of Lemma \ref{lem: svrgvar}, we have 
\begin{align*}
\norm{\Expect{\F_t \z\otimes \F_t \z} - \Expect{\F_t \z}\otimes \Expect{\F_t \z}} \le 4\theta \cdot p_t^2(\lambda_1;\beta)\left(\exp\left(\frac{4\norm{\Sigma} t}{\lambda_1^2 - 4\beta}\right) - 1\right)
\end{align*}
Further, if $4\norm{\Sigma} t \le \lambda_1^2 - 4\beta$, we have
\begin{align*}
\norm{\Expect{\F_t \z\otimes \F_t \z} - \Expect{\F_t \z}\otimes \Expect{\F_t \z}} \le  p_t^2(\lambda_1;\beta)\frac{32\theta \norm{\Sigma} t}{\lambda_1^2 - 4\beta}.
\end{align*}
\end{restatable}
\begin{proof}
The proof is the same as the one of Corollary \ref{cor: covbound}. 
\end{proof}

\begin{restatable}{lemma}{lemsvrgepoch}\label{lem: svrgepoch}
Suppose we run Algorithm \ref{alg: vr_power_m} for one epoch of length $t$ with initial unit vector $\w_0$. Assume that $\theta = 1 - (\u_1^T\w_0)^2< 1/2$ is small. Under the same condition of Lemma \ref{lem: svrgvar}, for any $\delta \in(0,1)$, when 
\begin{align*}
t &= \frac{\sqrt{\beta}}{\sqrt{\lambda_1^2 - 4 \beta}}
  \log \left( \frac{1}{\delta  c  } \right)
  \\
\|\Sigma \| &\le \frac{(\lambda_1^2 - 4 \beta)^{3/2} \delta  c }{32 \sqrt d \sqrt{\beta} }
  \log^{-1} \left( \frac{1}{\delta  c } \right).
\end{align*}
then with probability at least $1- 2\delta$, we have
\begin{align*}
1 - \frac{(\u_1^T\w_t)^2}{\norm{\w_t}^2} \le \frac{1}{9} \cdot (1 - (\u_1^T\w_0)^2).
\end{align*}
where $ c \in(0,1/16)$ is some small constant.
\end{restatable}
\begin{proof}
First, 
\begin{align*}
\Var{\u_1^T\F_t \w_0} &= (\u_1\otimes \u_1)^T (\Expect{\F_t \w_0\otimes \F_t \w_0} - \Expect{\F_t \w_0}\otimes \Expect{\F_t \w_0})
\\ & \le \norm{\u_1\otimes \u_1}_2\norm{\Expect{\F_t \w_0\otimes \F_t \w_0} - \Expect{\F_t \w_0}\otimes \Expect{\F_t \w_0}}_2
\\ & \le 4\theta \cdot p_t^2(\lambda_1;\beta)\frac{32\theta \norm{\Sigma} t}{\lambda_1^2 - 4\beta}
\end{align*}
by Chebyshev's inequality, for any $\delta > 0$, we have
\[
  \Prob{
    \Abs{ \u_1^T \F_t \w_0 - p_t(\lambda_1; \beta) \u_1^T \w_0 }
    \ge
    \frac{1}{\sqrt{\delta}}
    \cdot
    p_t\left( \lambda_1; \beta \right)
    \cdot
    \sqrt{ \frac{32 \theta \norm{\Sigma} t }{(\lambda_1^2 - 4 \beta)} }
  }
  \le
  \delta.
\]
That is,
\[
  \Prob{
    \Abs {\u_1^T \F_t \w_0}
    \le
    p_t(\lambda_1; \beta) \left(
      \Abs{\u_1^T \w_0}
      -
      \sqrt{\frac{32 \theta \norm{\Sigma} t }{(\lambda_1^2 - 4 \beta) \delta}}
    \right)
  }
  \le
  \delta.
\]
On the other hand,
\begin{align*}
  \sum_{i=2}^d \Expect{(\u_i^T \F_t \w_0)^2}
  &=
  \sum_{i=2}^d \left(
    \Expect{(\u_i^T \F_t \w_0)^2} - \Expect{\u_i^T \F_t \w_0}^2 + \Expect{\u_i^T \F_t \w_0}^2
  \right)
  \\
  & = \sum_{i=2}^d \left[\Expect{(\u_i^T\F_t\w_0)^{\otimes 2}} - \Expect{\u_i^T\F_t\w_0}^{\otimes 2} \right] + \sum_{i=2}^d p_t^2(\lambda_i;\beta)(\u_i^T\w_0)^2
  \\
  & = (\sum_{i=2}^d \u_i^{\otimes 2})^T\cdot \left[\Expect{(\F_t\w_0)^{\otimes 2}} - \Expect{\F_t\w_0}^{\otimes 2} \right] + \sum_{i=2}^d p_t^2(\lambda_i;\beta)(\u_i^T\w_0)^2
  \\ &\le
  \sqrt d \cdot \left(
    p_t^2\left( \lambda_1; \beta \right)
    \cdot
    \frac{32\theta \norm{\Sigma} t }{(\lambda_1^2 - 4 \beta)}
  \right)
  +
  \theta p_t^2(2\sqrt\beta; \beta)
  \\ 
  &\le
  \sqrt d \cdot p_t^2\left( \lambda_1; \beta \right)
  \cdot
  p_t^2\left( \lambda_1; \beta \right)
  \cdot
  \frac{32 \theta \norm{\Sigma} t }{(\lambda_1^2 - 4 \beta)}
  +
  \theta p_t^2(2 \sqrt{\beta}; \beta)
  \\ 
  &\le\theta \cdot
  p_t^2\left( \lambda_1; \beta \right) \left(
    \frac{32 \sqrt d \norm{\Sigma} t }{\lambda_1^2 - 4 \beta}
    +
    \frac{p_t^2(2 \sqrt{\beta}; \beta) }{p_t^2\left( \lambda_1; \beta \right) }
  \right).
\end{align*}
Therefore, by Markov's inequality,
\[
  \Prob{
    \sum_{i=2}^n (\u_i^T \F_t \w_0)^2
    \ge
    \theta \cdot p_t^2\left( \lambda_1; \beta \right) \left(
      \frac{32 \sqrt d \norm{\Sigma}t }{(\lambda_1^2 - 4 \beta) \delta}
      +
      \frac{p_t^2(2 \sqrt{\beta}; \beta) }{ \delta p_t^2\left( \lambda_1; \beta \right) }
    \right)
  }
  \le
  \delta.
\]
It follows by a union bound that
\[
  \Prob{
    \frac{
      \sum_{i=2}^d (\u_i^T \F_t \w_0)^2
    }{
      (\u_1^T \F_t \w_0)^2
    }
    \ge
    \theta
    \left(
      \frac{32\sqrt d \norm{\Sigma} t }{(\lambda_1^2 - 4 \beta) \delta}
      +
      \frac{ p_t^2(2 \sqrt{\beta}; \beta) }{ \delta p_t^2\left( \lambda_1; \beta \right) }
    \right)
    \left(
      \Abs{\u_1^T \w_0}
      -
      \sqrt{\frac{32\theta \norm{\Sigma} t }{(\lambda_1^2 - 4 \beta)\delta}}
    \right)^{-2}
  }
  \le
  2 \delta.
\]
Since $\Abs{\u_1^T\w_0}^2 \ge 1 - \theta $, then
\[
  \Prob{
    \frac{
      \sum_{i=2}^d (\u_i^T \F_t \w_0)^2
    }{
      (\u_1^T \F_t \w_0)^2
    }
    \ge
    \theta\cdot\left(
      \frac{32 \sqrt d \norm{\Sigma} t }{(\lambda_1^2 - 4 \beta) \delta}
      +
      \frac{ p_t^2(2 \sqrt{\beta}; \beta) }{ \delta p_t^2\left( \lambda_1; \beta \right) }
    \right)
    \left(
      \sqrt{1 - \theta }
      -
      \sqrt{\frac{32\theta \norm{\Sigma} t }{(\lambda_1^2 - 4 \beta)  \delta}}
    \right)^{-2}
  }
  \le
  2 \delta.
\]
For any $ c  \in (0,1/16)$, when 
\[
t = \frac{\sqrt\beta}{\sqrt{\lambda_1^2 - 4\beta}}\log(\frac{1}{\delta c }),
\]
we have 
\[
\frac{p_t^2(2\sqrt\beta;\beta)}{\delta p_t^2(\lambda_1;\beta)} \le  c .
\]
When
\[
\norm{\Sigma} \le \frac{(\lambda_1^2 - 4\beta)\delta c }{8\sqrt d t} =  \frac{(\lambda_1^2 - 4\beta)^{3/2}\delta c }{42\sqrt d\sqrt\beta}\log^{-1}\left(\frac{1}{\delta c }\right),
\]
we have
\[
\frac{8\sqrt{d}\norm{\Sigma} t}{\delta(\lambda_1^2 - 4\beta)} \le \epsilon.
\]
With both conditions, we have with probability at least $1-2\delta$, 
\begin{align*}
\frac{
      \sum_{i=2}^d (\u_i^T \F_t \w_0)^2
    }{
      (\u_1^T \F_t \w_0)^2
    }
    & \le \theta \cdot \frac{2 c }{\left(\sqrt{1-\theta} - \ \sqrt{\frac{ c \theta}{\sqrt d}}\right)^2}
    \\
    & \le \theta \cdot \frac{4 c }{(1 -\sqrt c )^2}
    \\
    & \le \frac{1}{9}\theta.
\end{align*}
The last two inequalities follow from the fact that $\theta < 1/2$ and $ c  < 1/16$.
Therefore with the conditions above, we have with probability at least $1-2\delta$, 
\begin{align*}
1 - \frac{(\u_1^T\w_t)^2}{\norm{\w_t}^2}  &\le \frac{1}{9}\cdot\theta.
\end{align*}
\end{proof}


\thmsvrg*
\begin{proof}
According to Lemma \ref{lem: svrgepoch}, if we have
\[
t = \frac{\sqrt{\beta}}{\sqrt{\lambda_1^2 - 4 \beta}}
  \log \left( \frac{1}{\delta  c  } \right),
  ~~~~
\|\Sigma \| \le \frac{(\lambda_1^2 - 4 \beta)^{3/2} \delta  c }{32 \sqrt d \sqrt{\beta} }
  \log^{-1} \left( \frac{1}{\delta  c } \right),
\]
 then with probability at least $1 - 2\delta$, 
\[
1 - (\u_1^T\tilde\w_{k+1})^2 \le \frac{1}{9} \left(1 - (\u_1^T\tilde\w_k)^2\right).
\]
holds.
In order to achieve $\epsilon$ accuracy, we run $K = \frac{\log(1/\epsilon)}{\log 9} = \bigO(\log(1/\epsilon))$ epochs and the success probability follows from a union bound which is $1 - 2\frac{\log(1/\epsilon)}{\log 9}\delta \ge 1 - \log(1/\epsilon)\delta$. 

Now use the fact that 
\[
\norm{\Sigma} =\norm{\Expect{(\A_t - \A)^{\otimes 2}}} \le \Expect{\norm{(\A_t - \A)^{\otimes 2}}} = \Expect{\norm{\A_t-\A}^2} = \frac{\sigma^2}{s},
\] and we get a sufficient condition on the batch size $s$, which is
\[
s \ge \frac{32\sqrt d\sqrt \beta\sigma^2}{c(\lambda_1^2 - 4\beta)\delta} \log\left(\frac{1}{\delta}\right).
\]
With that, it completes the proof.
\end{proof}

\section{Technical Lemmas}
This section contains the lemmas or statements that were used for the analysis in the appendix.

\begin{restatable}{lemma}{lemop}
\label{lem: op}
Given the polynomial sequence $\{p_t(x)\}$ defined in \eqref{alg: momentum},  when $\beta > 0$, we have
\begin{align*}
p_t(x) = \begin{cases} \frac{1}{2} \resizebox{4cm}{!}{$\left[\left( \frac{x - \sqrt{x^2 - 4\beta}}{2}\right)^{t} + \left( \frac{x+\sqrt{x^2-4\beta}}{2}\right)^{t}  \right]$}, &  |x| > 2\sqrt\beta,
 \\ (\sqrt\beta)^t \cos\left (t \arccos(\frac{x}{2\beta})\right), & |x| \le 2\sqrt\beta.
\end{cases}
\end{align*}
\end{restatable}
\begin{proof}
Consider the generating function of $\{p_t(x)\}$, $G(x, z)= \sum_{t=0}^\infty p_t(x)z^t, z\in\bbC$. And 
\begin{align}
&\sum_{t=1}^\infty p_{t+1} z^{t+1} = \sum_{t=1}^\infty x p_t z^{t+1} - \beta \sum_{t=1}^\infty p_{t-1} z^{t+1}
\\ & G(x,z) - p_0 -p_1z = xz (G(x, z) - p_0) - \beta z^2 G(x,z)
\\ & (\beta z^2 - xz +1) G(x, z) = p_0 + (p_1 - p_0x)z
\end{align}
Since $p_0 = 1, p_1 = x/2$, we have
\begin{align*}
G(x, z) = \frac{1 -  xz/2}{\beta z^2 - x z + 1}= \frac{1 - xz/2}{\beta(z -r_1)(z - r_2)},
\end{align*}
where $r_1, r_2\in \bbC$ are two roots of $\beta z^2 - x z + 1$.  
When $r_1 \ne r_2$, we have
\begin{align*}
G(x, z) &= \frac{1-xz/2}{\beta(r_1 - r_2)} \left( \frac{1}{r_2 - z} - \frac{1}{r_1 - z} \right)
\\ &= \frac{1-xz/2}{\beta (r_1 - r_2)} \sum_{t=0}^\infty  \left [ \left( \frac{1}{r_2}\right)^{t+1} -  \left( \frac{1}{r_1}\right)^{t+1}   \right]z^t
\\& = \sum_{t=0}^\infty \left [ \frac{1/r_2 - x/2}{\beta (r_1 - r_2)} \left( \frac{1}{r_2}\right)^{t} - \frac{1/r_1 - x/2}{\beta (r_1 - r_2)} \left( \frac{1}{r_1}\right)^{t}   \right] z^t.
\end{align*}

When $|x|\ge 2\sqrt{\beta}$, $r_{1,2} = \frac{x \pm \sqrt{x^2 - 4\beta}}{2\beta}$. Therefore, 
\begin{align*}
G(x, z) = \frac{1-xz/2}{\beta(z - r_1) (z - r_2)}  = \sum_{t=0}^\infty\frac{1}{2}\left[\left( \frac{1}{r_2}\right)^{t} + \left( \frac{1}{r_1}\right)^{t}  \right]z^t.
\end{align*}
By complex analysis theory, when $|z|< r_2$, $G(x, z)$ is well-defined. By comparing the coefficient of $z^t$, we get 
\begin{align*}
p_t(x) = \frac{1}{2}\left[\left( \frac{x - \sqrt{x^2 - 4\beta}}{2}\right)^{t} + \left( \frac{x+\sqrt{x^2-4\beta}}{2}\right)^{t}  \right].
\end{align*}	

When $|x| < 2\sqrt\beta$, $r_{1,2} = \frac{x\pm \ii \sqrt{4\beta - x^2}}{2\beta}$. Then
\begin{align*}
G(x, z) = \sum_{t=0}^\infty\frac{1}{2}\left[\left( \frac{1}{r_2}\right)^{t} + \left( \frac{1}{r_1}\right)^{t}  \right]z^t.
\end{align*}
Then $z < \frac{1}{\beta} = |r_{1,2}|$, $G(x,z)$ is well-defined. Then
\begin{align*}
p_t(x) = \frac{1}{2}\left[\left( \frac{x - \ii\sqrt{ 4\beta-x^2}}{2}\right)^{t} + \left( \frac{x+\ii\sqrt{4\beta-x^2}}{2}\right)^{t}  \right].
\end{align*}

When $|x| = 2\sqrt\beta$, $r_{1,2} = \frac{x}{2\beta}$. Suppose $x= 2\beta$, then
\begin{align*} 
G(x,z) &= \frac{1}{1 -\sqrt\beta z} = \sum_{t=0}^\infty (\sqrt{\beta})^t z^t.
\end{align*}
When $|z| \le 1/\sqrt\beta$, then $G(x,z)$ is well-defined. Then
\begin{align*}
p_t(2\sqrt\beta) = (\sqrt\beta)^t. 
\end{align*}
Similarly, if $x = - 2\sqrt\beta$, we have $p_t(-2\sqrt\beta) = (-\sqrt\beta)^t$.

Combine all three cases, we have
\begin{align*}
p_t(x) = \begin{cases} \frac{1}{2}\left[\left( \frac{x - \sqrt{x^2 - 4\beta}}{2}\right)^{t} + \left( \frac{x+\sqrt{x^2-4\beta}}{2}\right)^{t}  \right], &  |x| > 2\sqrt\beta,
\\ \frac{1}{2}\left[\left( \frac{x - \ii\sqrt{ 4\beta-x^2}}{2}\right)^{t} + \left( \frac{x+\ii\sqrt{4\beta-x^2}}{2}\right)^{t}  \right], & |x| \le 2\sqrt\beta. 
\end{cases}
\end{align*}
\end{proof}

\begin{fact}[Binomial Expansion of Matrices]\label{fact}
For any matrix $\A,\B\in\bbR^{n\times n}$, the binomial expansion $(\A + \B)^t$ has the following form,
\[
(\A + \B)^t = \sum_{j=0}^t \sum_{\k\in S_{t-j}^{j+1}}\A^{k_1}\prod_{i=2}^{j +1}\B\A^{k_i},
\]
$S_m^n$ denotes the set of vectors in $\bbN^n$ with entries that sum to $m$, i.e.
\[
S_m^n = \{ \k = (k_1,\cdots, k_n)\in\bbN^n | \sum_{i=1}^n k_i = m\}.
\]
\end{fact}

\begin{restatable}{lemma}{lemmapcp}\label{lem: pcp}
Given $k_1, \cdots, k_n,k_{n+1}\in\bbN$, and $z\ge 1$, then we have
\[
\prod_{i=1}^{n+1} U_{k_i}^2(z)  \le U_{\sum_{i=1}^{n+1} k_i+n}^2(z) \cdot\frac{1}{(z^2 - 1)^n}
\]
\end{restatable}
\begin{proof}
We prove it by induction. First base case when $n=0$, this is trivial.
Assume $n = t$ this inequality holds, and consider $n = t+1$, 
\begin{align*}
\prod_{i=1}^{t+2} U_{k_i}^2(z) &\le U_{k_{t+2}}^2(z) \cdot U_{\sum_{i=1}^{t+1}(k_i+1)}^2(z) \cdot \frac{1}{(z^2 - 1)^t}
\\ & = (T_{k_{t+2}+1}^2(z) - 1) \cdot U_{\sum_{i=1}^{t+1}k_i+t}^2(z) \cdot \frac{1}{(z^2 - 1)^{t+1}}
\\ & \le T_{k_{t+2}+1}^2(z)\cdot U_{\sum_{i=1}^{t+1}k_i+t}^2(z) \cdot \frac{1}{(z^2 - 1)^{t+1}}
\\ & \le U_{\sum_{i=1}^{t+2}k_i+t+1}^2(z) \cdot \frac{1}{(z^2 - 1)^{t+1}},
\end{align*}
which completes the induction.
\end{proof}

\begin{lemma}\label{lem: ang}
Given any two unit vectors $\u, \v\in\bbR^d$, then we have
\begin{align*}
1 - (\u^T\v)^2 = \norm{(I - \u\u^T)\v}^2 = \norm{(I - \v\v^T)\u}^2.
\end{align*}
\end{lemma}
\begin{proof}
\begin{align*}
\norm{(I - \u\u^T)\v}^2 &= \norm{\v}^2 - 2(\v^T\u)^2 + (\u^T\v)^2
\\ & = 1 - (\v^T\u)^2.
\end{align*}
Similarly, we have $\norm{(I - \v\v^T)\u}^2 = 1 - (\v^T\u)^2$.
\end{proof}

\section{Data Generation for Figure \ref{fig: 1}}\label{a_sec: data}
The synthetic dataset $\X\in\bbR^{10^6\times 10}$ was just generated through its singular value decomposition. Specifically we fix a $10$ by $10$ diagonal matrix $\Sigma = \diag\{1, \sqrt{0.9},\cdots, \sqrt{0.9}\}$ and generate random orthogonal projection matrix $\U\in\bbR^{10^6\times 10}$ and random orthogonal matrix $\V\in\bbR^{10\times 10}$. And the dataset $X = \sqrt\n \U\Sigma \V^T$ which guarantees that the matrix $\A = \frac{1}{n}\X^T\X$ has eigen-gap $0.1$.

\end{document}